\documentclass[reqno]{amsart}

\usepackage{amsmath}
\usepackage{amsthm}
\usepackage{amsfonts}
\usepackage{amssymb}
\usepackage{enumerate}
\usepackage{graphicx}
\usepackage{todonotes}
\usepackage[nocompress]{cite}
\usepackage{hyperref}
\usepackage{subfigure}
\usepackage{comment}
\graphicspath{ {./images/} }

\newcommand{\indicator}[1]{\ensuremath{\mathbf{1}_{\{#1\}}}}
\newcommand{\oindicator}[1]{\ensuremath{\mathbf{1}_{{#1}}}}

\numberwithin{equation}{section}

\DeclareMathOperator{\cov}{Cov}

\DeclareMathOperator{\tr}{tr}

\renewcommand\Re{\operatorname{Re}}
\renewcommand\Im{\operatorname{Im}}
\newcommand{\eps}{\varepsilon}

\theoremstyle{plain}
\newtheorem{theorem}{Theorem}[section]

\newtheorem{proposition}[theorem]{Proposition}

\newtheorem{lemma}[theorem]{Lemma}
\newtheorem{corollary}[theorem]{Corollary}

\theoremstyle{definition}
\newtheorem{definition}[theorem]{Definition}

\newtheorem{remark}[theorem]{Remark}

\newcommand{\NN}{\mathbb{N}}

\newcommand{\RR}{\mathbb{R}}
\newcommand{\CC}{\mathbb{C}}

\newcommand{\ee}{\varepsilon}

\newcommand{\pp}{\varphi}
\newcommand{\EE}{\mathbb{E}}
\newcommand{\PP}{\mathbb{P}}

\newcommand{\sign}{\text{sign}}

\begin{document}
\title{Spectrum of L\'evy-Khintchine Random Laplacian Matrices}

\author[A. Campbell]{Andrew Campbell}
\address{Department of Mathematics, University of Colorado at Boulder, Boulder, CO 80309}
\email{andrew.j.campbell@colorado.edu}

\author[S. O'Rourke]{Sean O'Rourke}
\thanks{S. O’Rourke has been supported in part by NSF CAREER grant DMS-2143142.}
\address{Department of Mathematics, University of Colorado at Boulder, Boulder, CO 80309}
\email{sean.d.orourke@colorado.edu}
\begin{abstract}
We consider the spectrum of random Laplacian matrices of the form $L_n=A_n-D_n$ where $A_n$ is a real symmetric random matrix and $D_n$ is a diagonal matrix whose entries are equal to the corresponding row sums of $A_n$. If $A_n$ is a Wigner matrix with entries in the domain of attraction of a Gaussian distribution the empirical spectral measure of $L_n$ is known to converge to the free convolution of a semicircle distribution and a standard real Gaussian distribution. 

We consider real symmetric random matrices $A_n$ with independent entries (up to symmetry) whose row sums converge to a purely non-Gaussian infinitely divisible distribution, which fall into the class of L\'evy-Khintchine random matrices first introduced by Jung [Trans Am Math Soc, \textbf{370}, (2018)]. Our main result shows that the empirical spectral measure of $L_n$ converges almost surely to a deterministic limit. A key step in the proof is to use the purely non-Gaussian nature of the row sums to build a random operator to which $L_n$ converges in an appropriate sense. This operator leads to a recursive distributional equation uniquely describing the Stieltjes transform of the limiting empirical spectral measure. 
\end{abstract}

\maketitle

\section{Introduction}


We consider the empirical spectral measure\footnote{The definition for the empirical spectral measure and other notation used throughout is established in Subsection \ref{sec:Notation}.} of random Laplacian-type matrices of the form \begin{equation}\label{eq:LaplacianDef}
	L_n=A_n-D_n
\end{equation} where $A_n=(A_{ij})_{i,j=1}^{n}$ is an $n\times n$ real symmetric random matrix with independent entries up to symmetry, and $D_n$ is a diagonal matrix with $(D_n)_{ii}=\sum_{j=1}^n A_{ij}$. When $A_n$ is a Wigner matrix, i.e.\ $A_n$ has independent entries up to symmetry with mean zero and variance $\frac{1}{n}$, the empirical spectral measure of $A_n$ converges to Wigner's semicircle law, the empirical spectral measure of $D_n$ converges to a standard Gaussian distribution, and it was shown in \cite{BDJ06} that the empirical spectral measure of $L_n$ converges to the free convolution of the semicircle law and the standard real Gaussian measure. In this paper we will consider $A_n$ such that the diagonal entries of $D_n$ converge in distribution, not to the Gaussian distribution, but rather to a non-Gaussian infinitely divisible distribution. This model will include L\'evy matrices, sometimes referred to as heavy-tailed Wigner matrices, where the entries of $A_n$ are independent up to symmetry, but have infinite second moment, see Subsection \ref{subsection:Levy} for more details. Another important example arises when $A_n$ is the adjacency matrix of an Erd\H{o}s-R\'enyi random graph where the expected degree of any vertex remains fixed as the number of vertices goes to infinity. These $A_n$ fall into the class of L\'evy-Khintchine matrices, a generlization of L\'evy matrices defined by Jung in \cite{Jung16}, see Subsection \ref{subsection:Levy-Khintchine} for more on these matrices. 

The term Laplacian comes from graph theory, where the combinatorial Laplacian of a graph with vertex set $\{1,2,\dots,n\}$ is defined by\begin{equation}
	L_{ij}=\begin{cases}
		\deg(i),\,\text{ if } i=j\\
		-1,\,\text{ if } i\sim j\\
		0,\,\text{ if } i\nsim j,
	\end{cases} 
\end{equation} where $i\sim j$ if $\{i,j\}$ is an edge in the graph and $\deg(i)$ is the number of edges incident to a vertex $i$. The combinatorial Laplacian is the negative of what we refer to as the Laplacian. If the entries of $A_n$ are almost surely nonnegative, then $L_n$ is the infinitesimal generator of a (random) continuous time random walk and for this reason $L_n$ is referred to as a Markov matrix in some of the literature. We use the term Laplacian throughout. Spectral properties of real symmetric random Laplacian matrices have been studied in \cite{DingJiang2010,BDJ06,HuangLandon,Jiang2012,Chatterjee2021SpectralPF,MR3777782,MR2948689,MR3321497,MR4193182} and for non-symmetric random Laplacian matrices in \cite{BCC2014} when the entries of $A_n$ are in the domain of attraction of either a real or complex Gaussian random variable. Though because of the widespread use of graph Laplacians this list is incomplete. In these light-tailed cases the limiting spectral measure has a particularly nice free probabilistic interpretation (see \cite{MR3585560} for an introduction to free probability and random matrices). In \cite{BDJ06} Bryc, Dembo, and Jiang proved the following:
\begin{theorem}[Theorem 1.3 in \cite{BDJ06}]\label{thm:BDJ Limit}
	Let $\{X_{ij}:j\geq i\geq 1 \}$ be a collection of i.i.d.\ real random variables with $\EE X_{12}=0$ and $\EE X_{12}^2=1$, $X_{ij}=X_{ji}$ for $1\leq i\leq j$, and let $A_n=\left(X_{ij}/\sqrt{n}\right)_{i,j=1}^n$ be a random real symmetric matrix. With probability one, the empirical spectral measure of the matrix $L_n$ defined in \eqref{eq:LaplacianDef} converges weakly to the free additive convolution of the semicircle and standard Gaussian measures.
\end{theorem} The analogous free probabilistic limit was established in \cite{BCC2014} for non-symmetric $A_n$. While some of the above references study sparse Laplacian matrices, none consider random Laplacian matrices with heavy-tailed entries or sparse Laplacian matrices where the expected number of nonzero entries in a row is uniformly bounded in $n$.

Many of the tools and techniques we employ were developed in the study of heavy-tailed real symmetric, or \emph{L\'evy}, matrices by Bordenave, Caputo, and Chafa\"i in \cite{SymHeavy}. L\'evy matrices were introduced in \cite{PhysRevE.50.1810} as heavy-tailed versions of Wigner matrices. For the purposes of this paper an important distinction between L\'evy and Wigner matrices is that the row sums of a Wigner matrix converge in distribution to a Gaussian random variable, while the row sums of a L\'evy matrix converge to an $\alpha$-stable distribution for $0<\alpha<2$. The techniques in \cite{SymHeavy} were extended by Jung in \cite{Jung16} to random matrices whose row sums converge in distribution to an infinitely divisible distribution.

\subsection{L\'evy Matrices}\label{subsection:Levy} L\'evy matrices are the heavy-tailed analogue of Wigner matrices, where the entries are independent up to symmetry, but fail to have two finite moments.

\begin{definition}\label{def:LevyMatrices}
	A real symmetric random matrix $X$ is a \textit{L\'evy Matrix} if the diagonal entries are zero, the entries above the diagonal are independent and identically distributed (i.i.d.)\ copies of a real random variable $\xi$, and there exists $\theta\in[0,1]$ and $\alpha\in(0,2)$ such that  \begin{itemize}
		\item[(i)] $\lim\limits_{t\rightarrow\infty}\frac{\PP(\xi\geq t)}{\PP(|\xi|\geq t)}=\theta,$
		\item[(ii)] $\PP(|\xi|\geq t)=t^{-\alpha}L(t)$ for all $t\geq 1$,  where $L$ is a slowly varying function, i.e.\ $L(tx)/L(t)\rightarrow1$ as $t\rightarrow\infty$ for any $x>0$.
	\end{itemize} 
\end{definition}

The conditions in Definition \ref{def:LevyMatrices} as the same conditions for $\xi$ to be in the domain of attraction of an $\alpha$-stable distribution. Unlike with Wigner matrices, the natural scaling on an $n\times n$ L\'evy matrix $X$ is not $\sqrt{n}$, but instead \begin{equation}\label{eq:scaling}
a_n:=\inf\left\{x:\PP(|X_{12}|>x)\leq n^{-1}\right\}.
\end{equation} For an $n\times n$ L\'evy matrix $X_n$, the matrix $A_n$ in equation \eqref{eq:LaplacianDef} will be defined as $A_n:=a_n^{-1}X_n$. We will refer to $A_n$ as a \textit{normalized} L\'evy matrix.
\subsection{L\'evy-Khintchine Matrices}\label{subsection:Levy-Khintchine}

Jung in \cite{Jung16} defined a generalization of L\'evy matrices. Instead of assuming the entries are in the domain of attraction of an $\alpha$-stable distribution, the entries are in the domain of attraction of any infinitely divisible distribution.

\begin{definition}\label{def:LevyKhintchine}
	A sequence of real symmetric random matrices $\{A_n\}_{n\geq 1}$ is called a \emph{L\'evy-Khintchine} random matrix ensemble with characteristics $(\sigma^2,b,m)$ if for each $n$, $A_n=(A_{ij}^{(n)})_{i,j=1}^n$ is $n\times n$, the diagonal entries of $A_n$ are $0$, the non-diagonal entries are i.i.d.\ up to symmetry and $\sum_{j=1}^n A_{1j}^{(n)}$ converges in distribution as $n\rightarrow\infty$ to a random variable $Y$ with \begin{equation}\label{eq:IDcond}
	\log(\EE e^{itY})=-\frac{1}{2}t^2\sigma^2+itb+\int_\RR \left( e^{itx}-1-\frac{itx}{1+x^2}\right) dm(x)
	\end{equation} for all $t\in\RR$, where $m$ is a measure on $\RR$ with $m(\{0\})=0$ satisfying \begin{equation}\label{eq:squaresumcond}
	\int_\RR 1\wedge|x|^2dm(x)<\infty.
	\end{equation}
\end{definition}

\begin{remark}
	It is worth noting the that distribution of $A_{12}^{(n)}$ may change with $n$. However, for many important examples $A_{12}^{(n)}$ is either a rescaling of a fixed random variable or is the product of a fixed random variable and a Bernoulli random variable where only the Bernoulli random variable is changing with $n$.
\end{remark}

A random variable $Y$ satisfying (\ref{eq:IDcond}) is said to have an infinitely divisible distribution with characteristics $(\sigma^2,b,m)$ and \eqref{eq:IDcond} is referred to as the L\'evy-Khintchine representation of $Y$. When $\sigma=0$, $Y$ is called purely non-Gaussian and has an important connection to Poisson point processes with intensity measure $m$ outlined in Propositions \ref{prop:IDconv} and \ref{OrderStatConv}. 

\subsection{Notation}\label{sec:Notation} Throughout this paper we use $\Rightarrow$ to denote weak convergence of probability measure, convergence in distribution of random variables, and vague convergence of finite measures. 
For an $n\times n$ real symmetric  matrix $M$ the eigenvalues will always be considered in non-increasing order $\lambda_1\geq\lambda_2\geq\dots\geq\lambda_n$. We define the empirical spectral measure of an $n\times n$ real symmetric matrix $M$ to be the probability measure\begin{equation}
	\mu_{M}=\frac{1}{n}\sum_{i =1}^n\delta_{\lambda_i},
\end{equation} where $\delta_x$ is the Dirac delta measure at $x$.

A coupling of two probability measures $\mu_1$ and $\mu_2$ is a random tuple $(X,Y)$ such that $X$ is $\mu_1$ distributed and $Y$ is $\mu_2$ distributed. The symbol $\overset{d}{=}$ will be used to denote equality in distribution of random variables and $\mathcal{L}(X)$ will be used to denote the distribution of a random variable $X$. For two complex-valued square integrable random variables $\xi$ and $\psi$, we define the covariance between $\xi$ and $\psi$ as $\cov(\xi,\psi):=\EE[(\xi-\EE\xi)\overline{(\psi-\EE\psi)}]$. 

Throughout we will consider Poisson point processes on $\bar{\RR}\setminus\{0\}$, the one point compactification of $\RR$ with the origin removed, with some intensity measure $m$. We will consider both finite and infinite measures, so for convenience we will denote the points of this process by $\{y_i\}_{i\geq 1}$ for general $m$ where $y_i=0$  for any $i$ greater than an appropriate (possibly identically infinite) Poisson random variable, and when considering a specific finite measure $m$ we will denote the points by $\{y_i\}_{i=1}^N$ for a Poisson random variable $N$.

For a topological space $E$, let $C_K(E)$ denote the set of real-valued continuous functions on $E$ with compact support. We will use $\CC_+$ to be the set of complex numbers with strictly positive imaginary part. For a probability measure $\mu$ on $\RR$ we define the function $s_\mu:\CC_+\rightarrow\CC_+$ by\begin{equation}\label{eq:StieltjesTransformDefinition}
	s_\mu(z)=\int_\RR\frac{1}{x-z}d\mu(x),
\end{equation} and refer to $s_\mu$ as the Stieltjes transform of $\mu$.

We will use asymptotic notation ($O, o, \Theta$, etc.) under the assumption that $n\rightarrow\infty$ unless otherwise stated. $X=O(Y)$ if $X\leq CY$ for an absolute constant $C>0$ and all $n \geq C$, $X=o(Y)$ if $X\leq C_nY$ for $C_n\rightarrow 0$, $X=\Theta(Y)$ if $cY\leq X\leq CY$ for absolute constants $C,c>0$ and all $n \geq C$, and $X\sim Y$ if $X/Y\rightarrow1$.

\section{Main results} Throughout we will assume $A_n=(A_{ij}^{(n)})_{i,j=1}^n$ is the $n$-th element of a L\'evy-Khintchine random matrix ensemble with characteristics $(0,b,m)$, $D_n$ is a diagonal matrix with $(D_n)_{ii}=\sum_{j=1}^n A_{ij}^{(n)}$, and \begin{equation}\label{eq:GenLaplacianDef}
L_n=A_n-D_n.
\end{equation}  

\begin{definition}
	Let $\{A_n\}_{n\geq 1}$ be a L\'evy-Khintchine random matrix ensemble with characteristics $(0,b,m)$ and for each $n\geq1$ let $V_{1n}\geq V_{2n}\geq\dots\geq V_{(n-1)n}$ be the order statistics of $\{|A_{2n}^{(n)}|,|A_{3n}^{(n)}|,\dots,|A_{nn}^{(n)}|\}$.   $\{A_n\}_{n\geq 1}$ satisfies Condition \textbf{C1} if:\begin{itemize}
		\item The Poisson point process with intensity measure $m$ is almost surely summable, which from Campbell's Formula (Lemma \ref{lemma:CampbellsFormula}) is a equivalent to \begin{equation}\label{eq:SumCond}
		\int_{\RR\setminus\{0\}}|x|\wedge 1 dm(x)<\infty.
		\end{equation}
		
		\item $(V_{jn})_{j\geq1}$ is almost surely uniformly integrable in $n$, i.e. \begin{equation}\label{eq:A:almost sure summability}
		\lim\limits_{k\rightarrow\infty}\sup_{n> k}\sum_{i=k+1}^{n} V_{i,n}=0,
		\end{equation} almost surely.
		
		\item There exists $\ee>0$ and $C>0$ such that \begin{equation}\label{eq:MeasureDecayAssumption}
		m(\{x\in\RR:|x|\geq t\})\leq Ct^{-\ee},
		\end{equation} and \begin{equation}\label{eq:Entry Tail Assumption}
		n\PP\left(|A_{12}^{(n)}|\geq t\right)\leq Ct^{-\eps}
	\end{equation} for all $t>1/4$ and for every $n\in\NN$.
	\end{itemize}
\end{definition}

\begin{remark}\label{remark:examples}
	Some interesting and important examples of random matrices satisfying condition \textbf{C1} include\begin{enumerate}
		\item[(i)]\label{example:levy} $A_n=a_n^{-1}X_n$ for a L\'evy matrix $X_n$ with $\alpha\in(0,1)$ and $a_n$ as defined in $\eqref{eq:scaling}$. In this case $m=m_\alpha$ where $m_\alpha$ is the measure on $\RR$ with density $$f(x)=\alpha |x|^{-(1+\alpha)}\left(\theta\indicator{x>0}+(1-\theta)\indicator{x<0} \right),$$ for $\alpha$ and $\theta$ as in Definition \ref{def:LevyMatrices}. \footnote{For $\alpha\in[1,2)$, \eqref{eq:SumCond} and \eqref{eq:A:almost sure summability} will not hold. For this reason we believe the case where $\alpha\in[1,2)$ would require different techniques than those used here.}
		
		\item[(ii)] The adjacency matrix $A_n$ of an Erd\H{o}s-R\'enyi random graph $G(n,p)$ with $np\rightarrow \lambda\in(0,\infty)$. In this case the row sums of $A_n$ converge to Poisson random variables and $m=\lambda\delta_1$.
		
		\item[(iii)] The matrix $A_n=\frac{1}{\sqrt{\lambda}}E_n\circ X_n$ where $E_n$ is the adjacency matrix of an Erd\H{o}s-R\'enyi random graph $G(n,p)$ with $np\rightarrow \lambda\in(0,\infty)$, $X_n$ is a chosen from the Gaussian Orthogonal Ensemble (GOE), and $\circ$ is the Hadamard product of matrices. In this case $m=\lambda G_\lambda$ where $G_\lambda$ is the centered Gaussian probability measure with variance $\frac{1}{\lambda}$.
	\end{enumerate}
\end{remark}

The first two points of Condition \textbf{C1} will be important for handling the diagonal entries of $L_n$. \eqref{eq:SumCond} implies a Poisson point process with intensity measure $m$ is almost surely summable, which is stronger than the almost sure square summability implied by \eqref{eq:squaresumcond}. \eqref{eq:A:almost sure summability} implies that the row sums converge to the sum of the Poisson point process with intensity measure $m$. The last point is a technical assumptions needed in the proof of the main theorem given below. Heuristically the last point of Condition \textbf{C1} states that the infinitely divisible random variable $Y$ in Definition \ref{def:LevyKhintchine} has at least $t^{-\eps}$ tail decay, and this tail assumption holds entry-wise uniformly in $n$. The assumption in \eqref{eq:Entry Tail Assumption} is technical and used to prove tightness of the empirical spectral measures, but perhaps is not necessary and there may be room for refinement. Those choice of $1/4$ in the final condition is arbitrary, any positive constant would be sufficient.

\begin{theorem}[Eigenvalue Convergence for Laplacian L\'evy-Khintchine matrices]\label{thm:ESMConv}
	Let $\{A_n\}_{n\geq 1}$ be a L\'evy-Khintchine random matrix ensemble with characteristics $(0,b,m)$ all defined on the same probability space satisfying Condition \textbf{C1}, and for every $n\in\NN$ let $L_n$ be defined by \eqref{eq:GenLaplacianDef}. Then there exists a deterministic probability measure $\mu_m$ depending only on $m$ such that a.s. $\mu_{L_n}$ converges weakly to $\mu_m$, as $n\rightarrow\infty$.
\end{theorem}

While the random matrices satisfying Condition \textbf{C1} may appear very different for different $m$, a general description of $\mu_m$ is available through its Stieltjes transform and a recursive distributional equation (RDE).  A recursive distributional equation is an equation of the form \begin{equation}\label{eq:BasicRDE}
	X\overset{d}{=}g((Y_1,X_1),(Y_2,X_2),\dots)
\end{equation} where $\{X_n\}_{n=1}^\infty$ are i.i.d.\ copies of $X$ and $\{Y_n\}_{n=1}^\infty$ is some sequence of random variables independent from $\{X_n\}_{n=1}^\infty$. While we do not use existing results from the literature we did find the survey \cite{AldousBandyopadhyay2005} and the unpublished manuscript \cite{Alsmeyer2012} helpful for better understanding RDEs and contraction arguments in proving uniqueness of solutions. We encourage the interested reader to begin there for more information on RDEs.

\begin{theorem}[Recursive Distributional Equation for Stieltjes Transform of $\mu_m$]\label{thm:RDEofStieltjes}
	Let $\mu_m$ be the limiting deterministic measure from Theorem \ref{thm:ESMConv} and let $s_m(z)=\int_\RR\frac{1}{x-z}d\mu_m(x)$ be the Stieltjes transform of $\mu_m$. Then for every $z\in\CC_+$, $s_m(z)=\EE s_\varnothing(z)$ where $s_\varnothing$ is the Stieltjes transform of a random probability measure. Moreover, the distribution of $s_\varnothing$ is the unique distribution on the space of Stieltjes transforms of probability measures such that \begin{equation}\label{eq:RDEofStieltjes}
	s_\varnothing(z)\overset{d}{=}-\left(z-\sum_{j=1}^\infty\frac{y_j}{s_j(z)y_j-1} \right)^{-1}\text{ for all } z\in\CC_+,
	\end{equation} where $\{y_j\}_{j\geq 1}$ is a Poisson point process with intensity measure $m$ and $\{s_j\}$ is a collection of i.i.d.\ copies of $s_\varnothing$ independent from the point process. 
\end{theorem}

Theorems \ref{thm:ESMConv} and \ref{thm:RDEofStieltjes} give that the limiting empirical spectral measure of $L_n$ is uniquely determined by a Poisson point process with intensity measure $m$. For the examples outlined in Remark \ref{remark:examples} we will now give some more explicit descriptions of the corresponding point processes. \begin{enumerate}
	\item[(i)] Let $E_1,E_2,\dots$ be a sequence of independent exponential random variables with mean $1$ and $\Gamma_k=E_1+\cdots+E_k$. Additionally let $\ee_1,\ee_2,\dots$ be a sequence of i.i.d.\ random variables such that $$\PP(\ee_1=1)=\theta=1-\PP(\ee_1=-1).$$ Then (see \cite{EmpPP} Proposition 2) the collection $\{\ee_k\Gamma_k^{-1/\alpha}\}_{k\geq1}$ is a Poisson point process with intensity measure $m_\alpha$, the measure arising for L\'evy matrices, example (i) in Remark \ref{remark:examples}.
	\item[(ii)] For the Laplacian of very sparse random graphs, discusses in Remark \ref{remark:examples} (ii), the Poisson point process is quite simple. Let $N$ be a Poisson random variable with mean $\lambda$ and for $k\geq 1$ define $y_k$ by\begin{equation*}
		y_k=\begin{cases}
			1,\, k\leq N,\\
			0,\, k>N
		\end{cases}.
	\end{equation*} Then $\{y_k\}_{k\geq 1}$ is a Poisson point process with intensity measure $\lambda\delta_1$.
	\item[(iii)] For a very sparse GOE matrix described in example (iii) in Remark \ref{remark:examples}, let $Y_1,Y_2,\dots$ be independent standard real Gaussian random variables, and let $N$ be a Poisson random variable with mean $\lambda$. Define \begin{equation*}
		y_k=\begin{cases}
			\frac{Y_k}{\sqrt{\lambda}},\, k\leq N\\
			0,\, k>N
		\end{cases}.
	\end{equation*} Then $\{y_k\}_{k\geq 1}$ is a Poisson point process with intensity measure $\lambda G_\lambda$. This example is explored a bit further in Theorem \ref{thm:RecoveringFreeConvolution} below.
\end{enumerate}

RDE \eqref{eq:RDEofStieltjes} can be written as \begin{equation}\label{eq:New RDE}
	s_\varnothing(z)\overset{d}{=}-\left(z+\sum_{j=1}^\infty y_j-\sum_{j=1}^\infty\frac{y_j^2s_j(z)}{s_j(z)y_j-1} \right)^{-1}.
\end{equation} If we consider a diagonal matrix $\tilde{D}_n$ independent from $A_n$ with independent entries $(\tilde{D}_{n})_{ii}\overset{d}{=}(D_n)_{ii}$ and the matrix $\tilde{L}_n=A_n-\tilde{D}_n$ the work below leading up to the existence of \eqref{eq:RDEofStieltjes} could be adapted in a straightforward way to arrive at the following corresponding RDE for $\tilde{L}_n$, \begin{equation}\label{eq:IndRDE}
s_\varnothing(z)\overset{d}{=}-\left(z+\sum_{j=1}^\infty \tilde{y}_j-\sum_{j=1}^\infty y_j^2s_j(z) \right)^{-1}.
\end{equation} where $\{\tilde{y}_j\}$ is an independent copy of the point process of $\{y_j\}$, independent of $\{s_j\}$. For light-tailed $A_n$, Theorem \ref{thm:BDJ Limit} gives that the limiting spectral measure of $L_n$ is the free additive convolution of the semicircle measure and the Gaussian measure. This is the same limiting spectral measure for $A_n-D_n$ for $A_n$ independent of $D_n$. In contrast, the differences between equations \eqref{eq:New RDE} and \eqref{eq:IndRDE} suggest that for L\'evy-Khintchine $A_n$, the dependence between $A_n$ and $D_n$ can be seen in the limiting measure $\mu_m$.

\subsection{Outline} In Sections \ref{sec:Op} and \ref{sec:PWIT} we define local convergence for operators on $\ell^2(V)$ for a countable set $V$ and use the measure $m$ to build a random operator $L$. In Section \ref{sec:LocConv} we show that $L_n$ converges locally in distribution to $L$, and then in Section \ref{sec:ResConv} we upgrade this to convergence of the empirical spectral measures. Finally in Section \ref{sec:RDE} we show the Stieltjes transform of the limiting empirical spectral measure can be described as the expected value of the unique solution to \eqref{eq:RDEofStieltjes}. In the appendices we prove almost sure tightness of the collection $\{\mu_{L_n}\}_{n\geq1}$ and list some technical lemmas. We end this section with two corollaries of Theorem \ref{thm:RDEofStieltjes}. The first is a continuity result for the map $m\mapsto\mu_m$. In the second we use \eqref{eq:RDEofStieltjes} to recover the free convolution of a semicircle and a standard Gaussian measure from the limiting empirical measure of very sparse random matrices. 

\subsection{Corollaries of Theorem \ref{thm:RDEofStieltjes}} The first corollary of Theorem \ref{thm:RDEofStieltjes} concerns continuity of the mapping $m\mapsto \mu_m$ where $\mu_m$ is the limiting measure of Theorem \ref{thm:ESMConv}. Uniqueness of the solution to the RDE in Theorem \ref{thm:RDEofStieltjes} is crucial to the proof of Corollary \ref{cor:continuity in m} below.

\begin{corollary}\label{cor:continuity in m}
	Let $\bar{\RR}$ denote the one point compactification of $\RR$. Let $\{m_n\}_{n=1}^\infty$ be a collection of measures on $\RR$ such that\begin{equation*}
		\int_{\RR} \min(1,|x|)\, dm_n(x)<\infty,
	\end{equation*} for all $n\in\NN\cup\{\infty\}$, for any $f\in C_K(\bar{\RR}\setminus\{0\})$,\begin{equation*}
		\int_{\RR}fdm_n\rightarrow\int_\RR fdm_\infty,
	\end{equation*} and for any $\eps>0$ \begin{equation}\label{eq:convergence of sums}
	\lim\limits_{k\rightarrow\infty}\sup_{n\geq 1}\PP\left(\sum_{j=k}^\infty |y_j^{(n)}|>\eps \right)= 0,
\end{equation}  where for each $n\in\NN$, $\{y_j^{(n)}\}_{j=1}^\infty$ is a Poisson point process with intensity measure $m_n$. Then $\mu_{m_n}$ converges weakly to $\mu_{m_\infty}$ as $n\rightarrow\infty$, where $\mu_{m_1},\mu_{m_2},\dots$ and $\mu_{m_\infty}$ are the deterministic limiting measures described in Theorem \ref{thm:ESMConv} for a L\'evy-Khintchine random matrix ensemble with characteristics $(0,b,m_1), (0,b,m_2),\dots$ and $(0,b,m_\infty)$ respectively.
\end{corollary}
\begin{proof}
	Let $s_n$ be the Stieltjes transforms of $\mu_{m_n}$. Let $\{\mu_{m_{n_k}}\}_{n_k}$ be a subsequence of $\{\mu_{n}\}_{n}$, and let $r_{n_k}(z)$ be the random Stieltjes transforms solving RDE \eqref{eq:RDEofStieltjes} for the measures $m_{n_k}$. From Lemma \ref{tightness of random analytic functions} it follows that $\{r_{n_k}\}$ is tight in the space of analytic function on $\CC_+$ with the topology of uniform convergence on compact subsets, and we pass to a further subsequence $n_k'$ converging to another random analytic function $r(z)$. As $\{r_{n_k}\}$ is almost surely uniformly bounded on compact subsets is follows that $r$ is almost surely bounded on compact subsets. For any fixed $z\in\CC_+$, it follows by the dominated convergence theorem that \begin{equation}\label{eq:A:subseq convergence to r}
	\lim_{n_k'\rightarrow\infty}s_{n_k'}(z)=\lim_{n_k'\rightarrow\infty}\EE r_{n_k'}(z)=\EE r(z).
\end{equation} Corollary \ref{cor:continuity in m} then follows if $r$ is a random Stieltjes transform solution to RDE \eqref{eq:RDEofStieltjes} corresponding to $m_\infty$. 

To this end, let $\{\Pi_n\}_{n=1}^\infty$ be Poisson random measures with intensity measures $m_n$. For a positive function $f\in C_K(\bar{\RR}\setminus\{0\})$, $1-e^{-f(x)}$ is also a continuous function with compact support. Thus\begin{equation}
	\lim_{n\rightarrow\infty}\exp\left(\int_\RR 1-e^{-f(x)}dm_n(x) \right)=\exp\left(\int_\RR 1-e^{-f(x)}dm_\infty(x) \right).
\end{equation} It follows from Theorems 5.1 and 5.2 in \cite{HeavyPhenom} that $\Pi_n$ converges in distribution to $\Pi_\infty$. For $n\in\NN$ let $\{y_j^{(n)}\}_{j\geq 1}$ be the points of the process $\Pi_n$ and $\{y_j\}_{j\geq 1}$ the points of the process $\Pi_\infty$. The points may be ordered such that for every $j\in\NN$, $y_j^{(n)}$ converges in distribution to $y_j$ (see Section 2 of \cite{EmpPP} for more details). In fact, from \eqref{eq:convergence of sums} and Lemma 1 of \cite{EmpPP} $\{y_j^{(n)}\}_{j\geq 1}$ converges in distribution to $\{y_j\}_{j\geq 1}$ in $\ell^1(\NN)$.
Using Skorokhod's representation theorem we may put $\{r_{n_k'}\}$, $\{\Pi_{n_k'}\}$, $\Pi_\infty$ and $r$ on a single probability space such that all the above convergences in distributions are almost sure, and \begin{equation}\label{eq:l1 convergence}
	\lim\limits_{n\rightarrow\infty}\sum_{j=1}^\infty|y_j^{(n)}-y_j|=0
\end{equation} almost surely. For fixed $z\in\CC_+$, \begin{align}\label{eq:rconvergence}
r(z)&=\lim_{n_k'\rightarrow\infty}r_{n_k'}(z)\nonumber\\
	&=\lim_{n_k'\rightarrow\infty}-\left(z-\sum_{j=1}^\infty\frac{y_j^{(n_k')}}{r_{n_k'}^{(j)}(z)y_j^{(n_k')}-1} \right)^{-1}\nonumber\\
	&=-\left(z-\lim_{n_k'\rightarrow\infty}\sum_{j=1}^\infty\frac{y_j^{(n_k')}}{r_{n_k'}^{(j)}(z)y_j^{(n_k')}-1} \right)^{-1}\nonumber\\
	&=-\left(z-\sum_{j=1}^\infty\frac{y_j}{r_j(z)y_j-1} \right)^{-1},
\end{align} for independent copies $r_1,r_2,\dots$ of $r$, where the last equality follows from \eqref{eq:l1 convergence}. Thus $r$ is an analytic solution to RDE \eqref{eq:RDEofStieltjes}. From \eqref{eq:rconvergence} and the almost sure boundedness of $r$ on compact subsets of $\CC_+$ that almost surely\begin{equation}
\lim\limits_{t\rightarrow\infty}itr(it)=-1,
\end{equation} and thus $r$ is almost surely the Stieltjes transform of a probability measure. From \eqref{eq:A:subseq convergence to r} and the uniqueness of the solution to RDE \eqref{eq:RDEofStieltjes} it follows that for any $z\in\CC_+$ \begin{equation}
\lim_{n_k'\rightarrow\infty}s_{n_k'}(z)=s_{\infty}(z).
\end{equation} As the subsequence $n_k$ was arbitrary is follows that $s_n$ converge pointwise to $s_\infty$ and $\mu_{m_n}$ converges weakly to $\mu_{m_\infty}$ as $n\rightarrow\infty$.
\end{proof}

Theorem \ref{thm:RecoveringFreeConvolution} below considers the $\lambda\rightarrow\infty$ limit of example (iii) in Remark \ref{remark:examples}. The limiting measure is the same limiting measure found in Theorem \ref{thm:BDJ Limit}. The works of Jiang \cite{Jiang2012} and Chatterjee and Hazra \cite{Chatterjee2021SpectralPF} established Theorem \ref{thm:BDJ Limit} for sparse random matrices where the expected number of nonzero entries in a row tends to infinity with the size of the matrix. Theorem \ref{thm:RecoveringFreeConvolution}, when combined with Theorem \ref{thm:ESMConv} and Remark \ref{remark:examples} (iii), can then be interpreted as splitting the limit to where first $n\rightarrow\infty$ and then the expected number of nonzero entries tends to infinity.
\begin{theorem}\label{thm:RecoveringFreeConvolution}
	Let $G_\lambda$ denote the Gaussian probability measure with mean $0$ and variance $\frac{1}{\lambda}$, and let $m_\lambda=\lambda G_\lambda$. If $\mu_{m_\lambda}$ is the deterministic limiting probability measure from Theorem \ref{thm:ESMConv}, then $\mu_{m_\lambda}$ converges weakly to  the free convolution of the semicircle distribution and the standard real Gaussian distribution, as $\lambda\rightarrow\infty$.
\end{theorem}

\begin{proof}
	Denote the free convolution of a standard semicircle measure and standard Gaussian measure by  $\text{SC}\boxplus G_1$. It is known \cite{Biane97} the Stieltjes transform, $s_{\text{fc}}$, of $\text{SC}\boxplus G_1$ can be defined as the unique solution to \begin{equation}
		s_{\text{fc}}(z)=\int_\RR\frac{1}{x-z-s_{\text{fc}}(z)}\frac{1}{\sqrt{2\pi}}e^{-x^2/2}dx
	\end{equation} satisfying $\Im(s_{\text{fc}}(z))\geq 0$ and $s_{\text{fc}}(z)\sim-z^{-1}$ as $z\rightarrow\infty$. If $s_\lambda$ is the Stieltjes transform of $\mu_{m_\lambda}$, then from Theorem \ref{thm:RDEofStieltjes} we know $s_\lambda(z)=\EE r_\lambda(z)$ where $r_\lambda$ satisfies the RDE \begin{align}
	r_\lambda(z)\overset{d}{=}-\left(z-\sum_{j=1}^N\frac{y_j}{r_j(z)y_j-1} \right)^{-1},
\end{align} $N\sim\text{Pois}(\lambda)$, $\{y_j\}_{j=1}^\infty$ are i.i.d. Gaussian random variables with mean zero and variance $\frac{1}{\lambda}$ and $\{r_j\}_{j=1}^\infty$ are i.i.d. copies of $r_\lambda$, independent of the collection $\{y_j\}_{j=1}^\infty$. We will instead use the equivalent recursive distributional equation \begin{equation}
r_\lambda(z)\overset{d}{=}-\left(z+\frac{1}{\sqrt{\lambda}}\sum_{j=1}^Ny_j+\frac{1}{\lambda}\sum_{j=1}^N\frac{r_j(z)y_j^2}{1-r_j(z)y_j/\sqrt{\lambda}} \right)^{-1},
\end{equation} where $\{y_j\}_{j=1}^\infty$ are i.i.d. standard real Gaussian random variables. Fix $z\in\CC_+$. We first consider the sum $S_\lambda=\frac{1}{\sqrt{\lambda}}\sum_{j=1}^Ny_j$. For $t\in\RR$, define \begin{align*}
\pp_{S_\lambda}(t)&:=\EE\exp(itS_\lambda)\\
&=\sum_{k=0}^{\infty}\left(e^{-\frac{t^2}{2\lambda} }\right)^k\frac{\lambda^ke^{-\lambda}}{k!}\\
&=\exp\left({-\lambda}\right)\exp\left({\lambda e^{-\frac{t^2}{2\lambda}}}\right)\\
&=\exp\left(-\lambda+\lambda-\frac{t^2}{2}+o\left(\frac{1}{\lambda}\right)\right),
\end{align*} where here and throughout the proof asymptotic notation is as $\lambda\rightarrow\infty$. Thus $S_\lambda$ converges to a standard real Gaussian random variable as $\lambda\rightarrow\infty$.

We will compare the sum $\frac{1}{\lambda}\sum_{j=1}^N\frac{r_j(z)y_j^2}{1-r_j(z)y_j/\sqrt{\lambda}}$ to increasingly simpler sums. The first comparison is to the sum $\frac{1}{\lambda}\sum_{j=1}^Nr_{j}(z)y_j^2$. Notice that \begin{align*}
	\left|\frac{1}{\lambda}\sum_{j=1}^N\frac{r_j(z)y_j^2}{1-r_j(z)y_j/\sqrt{\lambda}}-\frac{1}{\lambda}\sum_{j=1}^Nr_{j}(z)y_j^2 \right|&=\left|\frac{1}{\lambda^{3/2}}\sum_{j=1}^N\frac{r_j(z)^2y_j^3}{1-r_j(z)y_j/\sqrt{\lambda}} \right|\\
	&\leq\frac{1}{\Im(z)^2\lambda^{3/2}}\sum_{j=1}^N\frac{|y_j|^3}{|1-r_j(z)y_j/\sqrt{\lambda}|}\\
	&\leq\frac{2}{\Im(z)^2\lambda^{3/2}}\sum_{j=1}^N|y_j|^3\\
	&\quad+\frac{1}{\Im(z)^2\lambda^{3/2}}\sum_{j=1}^N\frac{|y_j|^3}{|1-r_j(z)y_j/\sqrt{\lambda}|}\oindicator{A_{j,\lambda}},
\end{align*} where $\oindicator{A_{j,\lambda}}$ is the indicator of the event $A_{j,\lambda}=\{|y_j|\geq\sqrt{\lambda}\Im(z)/2\}$. We will now show both pieces of this bound converge in probability to zero. From Lemma \ref{lemma:CampbellsFormula} \begin{align*}
\lim\limits_{\lambda\rightarrow\infty}\EE\frac{1}{\lambda^{3/2}}\sum_{j=1}^N|y_j|^3=\lim\limits_{\lambda\rightarrow\infty}\frac{\EE|y_1|^3}{\sqrt{\lambda}}=0.
\end{align*} From standard tail estimates of Gaussian random variables we have that \begin{align*}
\lim\limits_{\lambda\rightarrow\infty}\PP\left(\sum_{j=1}^N\oindicator{A_{j,\lambda}}\neq 0 \right)&=\lim\limits_{\lambda\rightarrow\infty}\sum_{k=0}^\infty\PP\left(\sum_{j=1}^k\oindicator{A_{j,\lambda}}\neq 0 \right)e^{-\lambda}\frac{\lambda^k}{k!}\\
&\leq \lim\limits_{\lambda\rightarrow\infty}\sum_{k=0}^\infty k\PP\left(|y_1|\geq\sqrt{\lambda}\Im(z)/2\right)e^{-\lambda}\frac{\lambda^k}{k!}\\
&\leq \lim\limits_{\lambda\rightarrow\infty}\sum_{k=0}^\infty k\frac{2}{\sqrt{2\pi\lambda}\Im(z)}e^{-\lambda\Im(z)^2/8} e^{-\lambda}\frac{\lambda^k}{k!}\\
&\leq\lim\limits_{\lambda\rightarrow\infty}Ce^{-c\lambda}\sqrt{\lambda},
\end{align*} for some positive constants $C,c>0$ independent of $\lambda$. Thus $\frac{1}{\lambda}\sum_{j=1}^N\frac{r_j(z)y_j^2}{1-r_j(z)y_j/\sqrt{\lambda}}-\frac{1}{\lambda}\sum_{j=1}^Nr_{j}(z)y_j^2\Rightarrow0$ as $\lambda\rightarrow\infty$. Next we compare to the sum $\frac{1}{\lambda}\sum_{j=1}^N(\EE r_{\lambda}(z))y_j^2$. To this end let $Z_j=r_j(z)y_j^2-(\EE r_\lambda(z))y_j^2$, and consider the Taylor expansion of characteristic function of the real part of $\frac{1}{\lambda}\sum_{j=1}^NZ_j$ \begin{align*}
\lim\limits_{\lambda\rightarrow\infty}\EE\exp\left(it\frac{1}{\lambda}\sum_{j=1}^N\Re(Z_j)\right)&=\lim\limits_{\lambda\rightarrow\infty}\sum_{k=0}^\infty\left[\EE\exp\left(it\frac{1}{\lambda}\Re(Z_1)\right)\right]^ke^{-\lambda}\frac{\lambda^k}{k!}\\
&=\lim\limits_{\lambda\rightarrow\infty}\exp\left(-\lambda+\lambda+it\EE \Re(Z_1)+O(1/\lambda) \right)\\
&=1.
\end{align*} An identical argument follows from the imaginary part, and we see that $\frac{1}{\lambda}\sum_{j=1}^NZ_j$ converges in probability to zero. It is also straightforward to show $\frac{1}{\lambda}\sum_{j=1}^Ny_j^2\Rightarrow1$, and thus $\frac{1}{\lambda}\sum_{j=1}^N\EE( r_{\lambda}(z))y_j^2-\EE r_{\lambda}(z)$ converges in probability to zero. These three comparisons lead to \begin{align*}
\frac{1}{\lambda}\sum_{j=1}^N\frac{r_j(z)y_j^2}{1-r_j(z)y_j/\sqrt{\lambda}}-\EE r_\lambda(z)&=\frac{1}{\lambda}\sum_{j=1}^N\frac{r_j(z)y_j^2}{1-r_j(z)y_j/\sqrt{\lambda}}-\frac{1}{\lambda}\sum_{j=1}^Nr_{j}(z)y_j^2\\
			&\quad+\frac{1}{\lambda}\sum_{j=1}^Nr_{j}(z)y_j^2-\frac{1}{\lambda}\sum_{j=1}^N\EE(r_{\lambda}(z))y_j^2\\
			&\quad+\EE(r_{\lambda}(z))\left(\frac{1}{\lambda}\sum_{j=1}^Ny_j^2-1\right),
\end{align*} which converges in distribution to $0$. Since this limit is a constant, we may conclude that jointly\begin{equation}
\left(\frac{1}{\sqrt{\lambda}}\sum_{j=1}^Ny_j,\frac{1}{\lambda}\sum_{j=1}^N\frac{r_j(z)y_j^2}{1-r_j(z)y_j/\sqrt{\lambda}}-\EE r_\lambda(z)\right)\Rightarrow (Y,0),
\end{equation} where $Y$ is a standard Gaussian random variable.

Let $\{\lambda_n\}_{n=1}^\infty$ be an arbitrary increasing sequence of positive real numbers going to infinity and let $\{\lambda_{n_k}\}$ be an arbitrary subsequence. From Lemma \ref{tightness of random analytic functions} $\{r_{\lambda_{n_k}}\}_{n_k}$ is tight as a family of random analytic functions on $\CC_+$ with the topology of uniform convergence on compact subsets, and thus there exists a further subsequence $\lambda_{n_{k'}}$ such that $ r_{\lambda_{n_{k'}}}(z)\rightarrow \tilde{r}(z)$ for some  random analytic function $\tilde{r}$. Fix $z\in\CC_+$, it follows from the dominated convergence theorem that $\EE r_{\lambda_{n_{k'}}}(z)\rightarrow \EE\tilde{r}(z)=:r(z)$ for some deterministic limit $r(z)$. As $z\in\CC_+$ was arbitrary, it follows from the above convergence in distribution and the continuous mapping theorem that \begin{align*}
	r(z)&=\lim\limits_{n_{k'}\rightarrow\infty}\EE r_{\lambda_{n_{k'}}}(z)\\
		&=-\lim\limits_{n_{k'}\rightarrow\infty}\EE\left(z+\frac{1}{\sqrt{\lambda_{n_{k'}}}}\sum_{j=1}^Ny_j+\frac{1}{\lambda_{n_{k'}}}\sum_{j=1}^N\frac{r_{\lambda_{n_{k'}}}(z)y_j^2}{1-r_{\lambda_{n_{k'}}}(z)y_j/\sqrt{\lambda_{n_{k'}}}} \right)^{-1}\\
		&=\EE\frac{-1}{z+Y+r(z)}\\
		&=\int_\RR\frac{1}{x-z-r(z)}\frac{1}{\sqrt{2\pi}}e^{-x^2/2}dx,
\end{align*} pointwise on $\CC_+$. Thus $r(z)=s_{\text{fc}}(z)$ along every one of these further subsequences of $\{\lambda_{n_{k}}\}$, and $s_\lambda(z)=\EE r_\lambda(z)\rightarrow s_{\text{fc}}(z)$. By Lemma \ref{lemma:StieltjestoVague} this pointwise convergence of the Stieltjes transforms implies $\mu_{m_\lambda}$ converges weakly to $\text{SC}\boxplus G_1$ as $\lambda\rightarrow\infty$.
\end{proof}

The matrix $X_n$ in Remark \ref{remark:examples} (iii) has Gaussian entries, and for convenience we stated Theorem \ref{thm:RecoveringFreeConvolution} for the corresponding measure $\lambda G_\lambda$. However, the proof can be adapted in a straightforward way to the analogous measures corresponding to $X_n$ from Remark \ref{remark:examples} (iii) having entries with mean zero, variance $\frac{1}{\lambda}$, and three finite moments.

\section*{Acknowledgment}
The first author thanks Yizhe Zhu for pointing out reference \cite{Shirai}.

\section{Operators on $\ell^2(V)$}\label{sec:Op} Let $V$ be a countable set and let $\ell^2(V)$ denote the Hilbert space defined by the inner product 
$$\langle\phi,\psi\rangle:=\sum_{u\in V}\bar{\phi}_u\psi_u,\quad \phi_u=\langle\delta_u,\phi\rangle,$$
where $\delta_u$ is the unit vector supported on $u\in V$. Let $\mathcal{D}(V)$ denote the dense subset of $\ell^2(V)$ of vectors with finite support. Let $(w_{uv})_{u,v\in V}$ be a collection of real numbers with $w_{uv}=w_{vu}$ such that for all $u\in V$,
$$\sum_{v\in V}|w_{uv}|^2<\infty.$$
We then define a symmetric linear operator $A$ with domain $\mathcal{D}(V)$ by \begin{equation}\label{eq:OP}
\langle \delta_u,A\delta_v\rangle=\langle \delta_v,A\delta_u\rangle=w_{uv}.
\end{equation}

\begin{definition}[Local Convergence]
	Suppose $(A_n)$ is a sequence of bounded operators on $\ell^2(V)$ and $A$ is a linear operator on $\ell^2(V)$ with domain $\mathcal{D}(A)\supset\mathcal{D}(V)$. For any $u,v\in V$ we say that $(A_n,u)$ converges locally to $(A,v)$, and write 
	$$(A_n,u)\rightarrow(A,v),$$
	if there exists a sequence of bijections $\sigma_n:V\rightarrow V$ such that $\sigma_n(v)=u$ and, for all $\phi\in\mathcal{D}(V)$,
	$$\sigma_n^{-1}A_n\sigma_n\phi\rightarrow A\phi,$$
	in $\ell^2(V)$, as $n\rightarrow\infty$. 
\end{definition} 

Here we use $\sigma_n$ for the bijection on $V$ and the corresponding linear isometry defined in the obvious way. This notion of convergence is useful to random matrices for two reasons. First, we will make a choice on how to define the action of an $n\times n$ matrix on $\ell^2(V)$, and the bijections $\sigma_n$ help ensure the choice of location for the support of the matrix does not matter. Second, local convergence also gives convergence of the resolvent operator at the distinguished points $u,v\in V$. This comes down to the fact that local convergence is strong operator convergence, up to the isometries. See \cite{HeavyIId} for details.

\begin{theorem}[Theorem 2.2 in \cite{SymHeavy}]\label{Thm:ResfromLC}
	If $(A_n)_{n=1}^\infty$ and $A$ are self-adjoint operators such that $(A_n,u)$ converges locally to $(A,v)$ for some $u,v\in V$, then, for all $z\in\CC_+$,\begin{equation}
	\langle\delta_u,(A_n-z)^{-1}\delta_u\rangle\rightarrow\langle\delta_v,(A-z)^{-1}\delta_v\rangle
	\end{equation} as $n\rightarrow\infty$.
\end{theorem}

To apply this to random operators we say that $(A_n,u)\rightarrow(A,v)$ in distribution if there exists a sequence of random bijections $\sigma_n$ such that $\sigma_n^{-1}A_n\sigma_n\phi\rightarrow A\phi$ in distribution for every $\phi\in\mathcal{D}(V)$. 

\section{Poisson weighted infinite tree}\label{sec:PWIT}  Let $\rho$ be a positive Radon measure on $\RR\setminus\{0\}$. $\mathbf{PWIT}(\rho)$ is the random infinite weighted rooted tree defined as follows. The vertex set of the tree is identified with $\NN^f:=\bigcup_{k\in\NN\cup\{0\}}\NN^k$ by indexing the root as $\NN^0=\varnothing$, the offspring of the root as $\NN$ and, more generally, the offspring of some $v\in\NN^k$ as $(v1),(v2),\cdots\in\NN^{k+1}$. Define $T$ as the tree on $\NN^f$ with edges between parents and offspring. Let $\{\Xi_v\}_{v\in\NN^f}$ be independent realizations of a Poisson point process with intensity measure $\rho$. Let $\Xi_\varnothing=\{y_1,y_2,\dots\}$ be ordered such that $|y_1|\geq|y_2|\geq\cdots$ with the convention $y_i=0$ for all $i$ large enough\footnote{If $\rho(\RR\setminus\{0\})<\infty$ then the number of points in $\RR\setminus\{0\}$ is a Poisson random variable. By large enough we mean larger than this random variable.} if $\rho(\RR\setminus\{0\})<\infty$, and assign the weight $y_i$ to the edge between $\varnothing$ and $i$, assuming such an ordering is possible. More generally assign the weight $y_{vi}$ to the edge between $v$ and $vi$ where $\Xi_v=\{y_{v1},y_{v2},\dots \}$ and $|y_{v1}|\geq|y_{v2}|\geq\cdots,$ again with the convention $y_{vi}=0$ for all $i$ larger than $\Xi_v(\RR\setminus\{0\})$ if $\rho(\RR\setminus\{0\})<\infty$.

For a measure $m$ on $\RR\setminus\{0\}$ satisfying \eqref{eq:squaresumcond} and a realization of $\mathbf{PWIT}(m)$ define the linear operator $A$ on $\mathcal{D}(\NN^f)$ by the formulas \begin{equation}\label{eq:PWITLimitOp}
\langle\delta_v,A\delta_{vk}\rangle=\langle\delta_{vk},A\delta_v\rangle=y_{vk}
\end{equation}
and $\langle\delta_v,A\delta_u\rangle=0$ otherwise. From (\ref{eq:squaresumcond}) one can see that the points in $\Xi_v$ are almost surely square summable for every $v\in\NN^f$, and thus $A$ is a well defined linear operator on $\mathcal{D}(\NN^f)$, though is possibly unbounded on $\ell^2(\NN^f)$.

\subsection{Poisson weighted infinite tree with loops} The Poisson weighted infinite tree has been utilized in \cite{HeavyIId,HeavyMarkovOriented,SymHeavy,Jung16,campbell2020spectrum} to study the empirical spectral distribution of heavy-tailed random matrices by showing the random matrices converge to the operator defined by (\ref{eq:PWITLimitOp}) for an appropriate measure $m$. One key feature of those matrices is the diagonal elements are negligible when compared to the largest entries in a row or column. This will not be the case for the Laplacian matrix $L_n$, thus we will need to define an operator on a slightly modified graph.

Let $m$ be a measure on $\RR\setminus\{0\}$ such that \begin{equation}\label{eq:sumcond}
\int_{\RR\setminus\{0\}}|x|\wedge 1 dm(x)<\infty.
\end{equation} Define the Poisson weighted infinite tree with loops $\mathbf{PWITL}(m)$ as the random weighted graph with vertex set $\NN^f$ and edge set $E\cup\bigcup_{v\in\NN^f}\{v,v\}$ where $E$ is the edge set of $\mathbf{PWIT}(m)$. The weights on edges in $E$ of $\mathbf{PWITL}(m)$ are the weights on edges in $E$ of $\mathbf{PWIT}(m)$ while the weight on a loop $\{v,v\}$ is\begin{equation}\label{eq:diagweight}
y_{vv}:=-y_{ul}-\sum_{k=1}^\infty y_{vk},
\end{equation} where $ul=v$ if $v$ is not $\varnothing$ and the weight on $\{\varnothing,\varnothing\}$ is \begin{equation}
y_{\varnothing\varnothing}:=-\sum_{k =1}^\infty y_k.
\end{equation} (\ref{eq:sumcond}) is enough to guarantee $y_{vv}$ is a well-defined random variable, see Lemma \ref{lemma:CampbellsFormula}. Define the operator $L$ by \begin{equation}\label{eq:PWITLLimitOp}
\langle\delta_v,L\delta_{vk}\rangle=\langle\delta_{vk},L\delta_v\rangle=y_{vk}\quad\text{ and }\quad \langle\delta_v,L\delta_v\rangle=y_{vv},
\end{equation} and $\langle\delta_v,L\delta_u\rangle=0$ otherwise. In which case we say $L$ is the \emph{operator associated} to $\mathbf{PWITL}(m)$.

We will show the sequence $\{(L_n,1)\}_{n\geq 1}$ converges locally in distribution to $(L,\varnothing)$ where $L$ is the linear operator on $\ell^2(\NN^f)$ associated to the $\mathbf{PWITL}(m)$. 

\subsection{Self-adjointness} In this section we review and apply a criteria established by Bordenave, Caputo, and Chafa\"i in \cite{SymHeavy} for unbounded operators to be essentially self-adjoint. There are two minor issues which prevent immediately applying their results to the operator $L$ associated to $\mathbf{PWITL}(m)$. First is they consider operators with skeletons which are trees, and not trees with loops. This is easy to overcome. The second obstacle is in the application of the criteria they consider only point processes associated to $\alpha$-stable distributions and not more general infinitely divisible distributions. This is overcome by the establishment of Lemma \ref{lemma:Tau<1}.

\begin{proposition}[Lemma A.3 in \cite{SymHeavy}]\label{prop:selfadjointcriterion}
	Let $A$ be a linear operator on $\ell^2(\NN^f)$ defined by (\ref{eq:OP}). We say $u\sim v$ if $u=v, u=vk,$ or $v=uk$ for some $k\in\NN$. Assume $w_{uv}=0$ if $u\nsim v$. Suppose there exists a constant $\kappa>0$ and sequence of finite connected subsets $S_n\subset \NN^f$, such that $S_n\subset S_{n+1}$, $\NN^f=\bigcup_{n\in\NN} S_n$, and for every $n$ and $v\in S_n$,\begin{equation}
	\sum_{u\notin S_n:u\sim v}|w_{uv}|^2\leq\kappa.
	\end{equation} Then $A$ is essentially self-adjoint. 
\end{proposition}

\begin{proof}
	Proposition \ref{prop:selfadjointcriterion} is not stated identically to Lemma A.3 in \cite{SymHeavy}, however the only added assumption is that vertices are connected to themselves, so that the graph of the skeleton of $A$ is not a tree. The step in the proof given in \cite{SymHeavy} which uses the tree structure is the fact that if $v\in S_n$, $u\sim v$, and $v'\in S_n\setminus\{v\}$ then $u\nsim v$ which is also true for a tree with loops.
\end{proof}

\begin{proposition}[Proposition A.2 in \cite{SymHeavy}]\label{prop:MSelfadjoint}
	Let $m$ be a measure on $\RR\setminus \{0\}$ satisfying \eqref{eq:squaresumcond}.
	Let $\{\Xi_v \}_{v\in\NN^f}$ be a collection of Poisson point process on $\RR\setminus\{0\}$ with intensity measure $m$. Let $\Xi_\varnothing=\{y_1,y_2,\dots\}$ be ordered such that $|y_1|\geq|y_2|\geq\cdots$, and $\Xi_v=\{y_{v1},y_{v2},\dots \}$ be ordered such that $|y_{v1}|\geq|y_{v2}|\geq\cdots$ with the convention the $y_{vk}$ or $y_k$ are eventually zero if $m(\RR\setminus\{0\})<\infty$.
	 Additionally let $\{y_{vv}\}_{v\in\NN^f}$ be a collection of real random variables. Define the symmetric linear operator $A$ on $\ell^2(\NN^f)$ by 
	 $$\langle\delta_v,A\delta_{vk}\rangle=\langle\delta_{vk},A\delta_v\rangle=y_{vk},\quad\text{ and }\quad\langle\delta_v,A\delta_v\rangle=y_{vv}, $$ and $\langle \delta_u,A\delta_v\rangle=0$ otherwise. Then, with probability 1, $A$ is essentially self-adjoint.
\end{proposition}

\begin{proof}
	While Proposition \ref{prop:MSelfadjoint} may not initially appear to be Proposition A.2 in\cite{SymHeavy}, the proofs are identical as Lemma A.4 in \cite{SymHeavy} is extended to the setting considered here in Lemma \ref{lemma:Tau<1} below.
\end{proof}

\section{Local convergence for the Laplacian of L\'evy-Khintchine matrices}\label{sec:LocConv}

For an $n\times n$ matrix $M$, extend $M$ to a bounded operator on $\ell^2(\NN^f)$ as follows. For $1\leq i,j,\leq n,$ let $\langle \delta_i,M\delta_j\rangle=M_{ij}$. and $\langle\delta_u,M\delta_v\rangle=0$ otherwise.

\begin{theorem}\label{thm:LocalConv01}
	Let $L_n$ be the matrix defined by \eqref{eq:GenLaplacianDef} for $\{A_n \}_{n\geq 1}$, a L\'evy-Khintchine random matrix ensemble satisfying \textbf{C1} and $L$ the linear operator on $\ell^2(\NN^f)$ associated to $\mathbf{PWITL}(m)$. Then, in distribution, $(L_n,1)\rightarrow(L,\varnothing)$, as $n\rightarrow\infty$.
\end{theorem}

The rest of this section is devoted to the proof of Theorem \ref{thm:LocalConv01}. Before considering $(L_n,1)$ we begin by showing $(A_n,1)$ converges to $(L+D,\varnothing)$ where $D$ is a diagonal operator. This follows from the work of Jung in \cite{Jung16}, we include the proof to establish notation and for the convenience of the reader. We define a network as a graph with edge weights taking values in some normed space. To begin let $G_n$ be the complete network, without loops, on $\{1,\dots,n\}$ whose weight on edge $\{i,j\}$ equals $\xi^n_{ij}$ for some collection $(\xi^n_{ij})_{1\leq i< j\leq n}$ of random variables taking values in some normed space. Now consider the rooted network $(G_n,1)$ with the distinguished vertex $1$.
For any realization $(\xi_{ij}^n)$, and for any $B,H\in\NN$ such that $(B^{H+1}-1)/(B-1)\leq n$, we will define a finite rooted subnetwork $(G_n,1)^{B,H}$ of $(G_n,1)$   whose vertex set coincides with a $B$-ary tree of depth $H$. To this end we partially index the vertices of $(G_n,1)$ as elements in 
$$J_{B,H}:=\bigcup_{l=0}^H\{1,\dots,B\}^l\subset\NN^f,$$
the indexing being given by an injective map $\sigma_n$ from $J_{B,H}$ to $V_n:=\{1,\dots,n\}$. We set $I_\varnothing:=\{1\}$ and the index of the root $\sigma_n^{-1}(1)=\varnothing$. The vertex $v\in V_n\setminus I_\varnothing$ is given the index $(k)=\sigma_n^{-1}(v), 1\leq k\leq B$, if $\xi^n_{1,v}$ has the $k$-{th} largest norm value among $\{\xi_{1j}^n, j\neq 1 \}$, ties being broken by lexicographic order\footnote{To help keep track of notation in this section, note that $v=(w)\in V_n$ if $w\in J_{B,H}$ and $\sigma_n(w)=v$.}. This defines the first generation, and let $I_1$ be the union of $I_\varnothing$ and this generation. If $H\geq 2$ repeat this process for the vertex labeled $(1)$ on $V_n\setminus I_1$ to order $\{\xi_{(1)j}^n\}_{j\in V_n\setminus I_1}$ to get $\{11,12,\dots,1B\}$. Define $I_2$ to be the union of $I_1$ and this new collection. Repeat again for $(2),(3),\dots,(B)$ to get the second generation and so on. Call this vertex set $V_n^{B,H}=\sigma_n J_{B,H}$. 

For a realization $T$ of $\mathbf{PWITL}(m)$, recall we assign the weight $y_{vk}$ to the edge $\{v,vk\}$ and the weight $y_{vv}$ to the edge $\{v,v\}$. Then $(T,\varnothing)$ is a rooted network. Call $(T,\varnothing)^{B,H}$ the finite rooted subnetwork obtained by restricting $(T,\varnothing)$ to the vertex set $J_{B,H}$, and the edge set without the loops. If an edge is not present in $(T,\varnothing)^{B,H}$ assign the weight $0$. We say a sequence $(G_n,1)^{B,H}$, for fixed $B$ and $H$, converges in distribution, as $n\rightarrow\infty$, to $(T,\varnothing)^{B,H}$ if the joint distribution of the weights converges weakly. 

Let $\xi_{ij}^n=L_{ij}=A_{ij}^{(n)}$, where $L_{ij}$ is the $ij$-th entry of $L_n$ for $1\leq i<j\leq n$. We aim to show with the choice of weights $(\xi^n_{ij})_{1\leq i< j\leq n}$ that for fixed $B,H$ $(G_n,1)^{B,H}$ converges weakly to $(T,\varnothing)^{B,H}$.

Order the elements of $J_{B,H}$ lexicographically, i.e. $\varnothing\prec 1\prec 2\prec\cdots\prec B\prec 11\prec 12\prec\cdots\prec B\cdots B$. For $v\in J_{B,H}$ let $\mathcal{O}_v$ denote the offspring of $v$ in $(G_n,1)^{B,H}$. By construction $I_\varnothing=\{1\}$ and $I_v=\sigma_n\left(\bigcup_{w\prec v}\mathcal{O}_w \right)$, where $w\prec v$ must be strict in this union. Thus at every step of the indexing procedure we order the weights of neighboring edges not already considered at a previous step. Thus for all $v$,
$$(\xi_{\sigma_n(v),j }^n)_{j\notin I_v}\overset{d}{=}(\xi_{1j}^n)_{1< j\leq n-|I_v|}.$$  

Note that by independence, Proposition \ref{OrderStatConv} still holds if you take the sum of Dirac measures at the random variables over $\{1,\dots,n \}\setminus I$ for any fixed finite set $I$. Thus by Proposition \ref{OrderStatConv} the weights from a fixed parent to its offspring in $(G_n,1)^{B,H}$ converge weakly to those of $(T,\varnothing)^{B,H}$. By independence we can extend this to joint convergence. Recall $(G_n,1)^{B,H}$ is a complete graph and not a tree with loops. Thus it remains to show the edges in $(G_n,1)^{B,H}$ which were not considered in the sorting procedure converge to $0$. This was shown for heavy-tailed weights in \cite{SymHeavy} and for more general L\'evy-Khintchine weights in \cite{Jung16}.

Let $L$ be the operator associated to $\mathbf{PWITL}(m)$. For fixed $B,H$ let $\sigma_n^{B,H}$ be the map $\sigma_n$ above associated to $(G_n,1)^{B,H}$, and arbitrarily extend $\sigma_n^{B,H}$ to a bijection on $\NN^f$, where $V_n$ is considered in the natural way as a subset of the offspring of $\varnothing$.  From the Skorokhod representation theorem we may assume $(G_n,1)^{B,H}$ converges almost surely to $(T,\varnothing)^{B,H}$. Thus there are sequences $B_n,H_n$ tending to infinity and $\hat{\sigma}_n:=\sigma_n^{B_n,H_n}$ such that for any pair $v,w\in\NN^f$ with $w\neq v$, $\xi_{\hat{\sigma}_n(v),\hat{\sigma}_n(w)}^n$ converges almost surely to 
$$\begin{cases}
y_{vk},\quad \text{if $w=vk$ for some $k$}\\
y_{wk},\quad \text{if $v=wk$ for some $k$}\\
0,\quad\text{otherwise.}
\end{cases}$$ 
Thus for any $u,v\in \NN^f$ with $u\neq v$\begin{equation}\label{eq:nondiagconv}
\langle \delta_u, \hat{\sigma}_n^{-1}L_n\hat{\sigma}_n\delta_v\rangle=\langle \delta_u, \hat{\sigma}_n^{-1}A_n\hat{\sigma}_n\delta_v\rangle\rightarrow\langle\delta_u,(L+D)\delta_v\rangle=\langle\delta_u,L\delta_v\rangle
\end{equation} almost surely. We now consider the diagonal elements. Let $u\in\NN^f$, $B=H=k$ for some $k\in\NN$ such that $u\in J_{k,k}$. From the above we know almost surely \begin{equation}
\sum_{v\in J_{k,k}} \xi^n_{(u),(v)}\rightarrow\sum_{v\in J_{k,k}, v\sim u} y_v.
\end{equation} Assume $v\notin J_{k,k}$, then $\xi^n_{(u),(v)}\rightarrow 0$ almost surely. By the uniform summability condition of \textbf{C1} we have almost surely \begin{equation}
\sum_{v\notin J_{k,k}} \xi^n_{(u),(v)}\rightarrow0.
\end{equation} As $k$ was arbitrarily large we have that almost surely for any $v\in\NN^f$ \begin{equation}
\langle\delta_v,\hat{\sigma}_n^{-1}L_n\hat{\sigma}_n\delta_v\rangle\rightarrow\langle\delta_v,L\delta_v\rangle.
\end{equation} From linearity it suffices to show for every $v\in\NN^f$ that $\hat{\sigma}_n^{-1}L_n\hat{\sigma}_n\delta_v\rightarrow L\delta_v$, i.e.\begin{equation}\label{eq:l2conv}
\sum_{u\in\NN^f}\left[\langle\delta_u,\hat{\sigma}_n^{-1}L_n\hat{\sigma}_n\delta_v\rangle-\langle\delta_u,L\delta_v\rangle \right]^2\rightarrow 0.
\end{equation} We have shown $\langle\delta_u,\hat{\sigma}_n^{-1}L_n\hat{\sigma}_n\delta_v\rangle\rightarrow\langle\delta_u,L\delta_v\rangle$ almost surely for every $u\in\NN^f$, thus \eqref{eq:l2conv} holds if $\left\{\langle\delta_u,\hat{\sigma}_n^{-1}L_n\hat{\sigma}_n\delta_v\rangle \right\}_{u\in\NN^f}$ is uniformly square-summable. This follows from the uniform summability of \textbf{C1}. This completes the proof of Theorem \ref{thm:LocalConv01}.

We will need the following extension of Theorem \ref{thm:LocalConv01}.
\begin{theorem}\label{thm:LocalConv2}
	Let $L_n$ be the matrix defined by \eqref{eq:GenLaplacianDef} for $\{A_n \}_{n\geq 1}$, a L\'evy-Khintchine random matrix ensemble satisfying \textbf{C1}. If $L$ and $L'$ are two independent copies of the linear operator on $\ell^2(\NN^f)$ associated to $\mathbf{PWITL}(m)$, then, in distribution, $(L_n\oplus L_n,(1,2))\rightarrow(L\oplus L',(\varnothing,\varnothing))$ as $n\rightarrow\infty$.
\end{theorem}

\begin{proof}
	Using Proposition 2.6 in \cite{SymHeavy} and the arguments above we can construct isometries $\sigma_n$ on $\ell^2(\NN^f)\oplus\ell^2(\NN^f)$ such that for any $v\in\NN^f,\ \sigma_n^{-1}(L_n\oplus L_n)\sigma_n(\delta_v,0)\rightarrow L\oplus L'(\delta_v,0)$ and $\sigma_n^{-1}(L_n\oplus L_n)\sigma_n(0,\delta_v)\rightarrow L\oplus L'(0,\delta_v)$ in $\ell^2(\NN^f)\oplus\ell^2(\NN^f)$ almost surely. The result then follows by linearity.
\end{proof}

\section{Resolvent convergence and the proof of Theorem \ref{thm:ESMConv}}\label{sec:ResConv} 

\begin{theorem}\label{thm:ExpectedStieltjes}
	Let $s_{L_n}(z)$ be the Stieltjes transform of $\mu_{L_n}$ and let $s_{\varnothing}(z)$ be the Stieltjes transform of the measure $\mu_\varnothing$ defined by \begin{equation}
	\langle\delta_\varnothing,f(L)\delta_\varnothing\rangle=\int_\RR fd\mu_\varnothing
	\end{equation} for any continuous bounded function $f:\RR\rightarrow\CC$, where $f(L)$ is defined by the continuous functional calculus. Then \begin{equation}
	\lim\limits_{n\rightarrow\infty}\EE s_{L_n}(z)=\EE s_{\varnothing}(z)
	\end{equation} for every $z\in\CC_+$.
\end{theorem}

\begin{proof}
	For $z\in\CC_+$ we define the operators \begin{equation}
	R_n(z)=(L_n-z)^{-1},
	\end{equation} and \begin{equation}
	R(z)=(L-z)^{-1}.
	\end{equation} Additionally for $u,v\in\NN^f$, we define the functions $R_n(z)_{uv},R(z)_{uv}:\CC_+\rightarrow\CC$ by\begin{equation}
	R_n(z)_{uv}:=\langle\delta_u,R_n(z)\delta_v\rangle,\quad\text{ and }\quad R(z)_{uv}:=\langle\delta_u,R(z)\delta_v\rangle.
	\end{equation}
	From Proposition \ref{prop:MSelfadjoint}, $L$ is  self-adjoint with probability $1$. Thus from Theorem \ref{Thm:ResfromLC} and Theorem \ref{thm:LocalConv01} \begin{equation}
	R_n(z)_{11}=\langle\delta_1,(L_n-z)^{-1}\delta_1\rangle\Rightarrow\langle\delta_\varnothing,(L-z)^{-1}\delta_\varnothing\rangle=R(z)_{\varnothing\varnothing}.
	\end{equation} For every $z\in\CC_+$, $R_n(z)_{11}$ and $R(z)_{\varnothing\varnothing}$ are bounded, thus \begin{equation}
	\lim\limits_{n\rightarrow\infty}\EE R_n(z)_{11}=\EE R(z)_{\varnothing\varnothing}.
	\end{equation} By definition $s_\varnothing(z)=R(z)_{\varnothing\varnothing}$, while \begin{equation}
	s_{L_n}(z)=\frac{1}{n}\tr(L_n-z)^{-1}=\frac{1}{n}\sum_{i =1}^n R_n(z)_{ii}.
	\end{equation} It is clear from the matrix of cofactors method of inversion $R_n(z)_{ii}\overset{d}{=}R_n(z)_{jj}$ for every $i,j\in[n]$. Thus \begin{align*}
	\EE s_{L_n}(z)&=\frac{1}{n}\sum_{i =1}^n\EE R_n(z)_{ii}\\
				  &=\frac{1}{n}\sum_{i =1}^n\EE R_n(z)_{11}\\
				  &=\EE R_n(z)_{11}.
	\end{align*} This completes the proof.
\end{proof}

\subsection{Proof of Theorem \ref{thm:ESMConv}}

We are now ready to complete the proof of Theorem \ref{thm:ESMConv}. From Lemma \ref{lemma:tightness}, $\{\mu_{L_n}\}_{n\geq 1}$ is almost surely tight. Consider the Stieltjes transfrom $s_{L_n}$ of $\mu_{L_n}$. From tightness and Lemma \ref{lemma:StieltjestoVague} it is enough to prove that almost surely there exists the Stieltjes transform, $s$, of a probability measure such that for any subsequence $\{n_k\}$  \begin{equation}
	\lim_{n_k\rightarrow\infty}s_{L_{n_k}}(z)=s(z),
\end{equation} for all $z\in\CC_+$. We know from Theorem \ref{thm:ExpectedStieltjes} that for all $z\in\CC_+$\begin{equation}
\lim\limits_{n\rightarrow\infty}\EE s_{L_n}(z)=\EE s_{\varnothing}(z).
\end{equation} We now upgrade this to almost surely convergence of $s_{L_n}(z)$ to $\EE s_{\varnothing}(z)$. For $z\in\CC_+$\begin{equation*}
\EE|s_{L_n}(z)-\EE s_{\varnothing}(z)|\leq\EE|s_{L_n}(z)-\EE s_{L_n}(z)|+|\EE s_{L_n}(z)-\EE s_{\varnothing}(z)|.
\end{equation*} For $z\in\CC_+$\begin{equation}
\EE|s_{L_n}(z)-\EE s_{L_n}(z)|=\EE\left|\frac{1}{n}\sum_{i =1}^n\left( R_n(z)_{ii}-\EE R_n(z)_{ii}\right) \right|,
\end{equation} and by the exchangeability of the matrix entries\begin{align*}
\EE\left[\left|\frac{1}{n}\sum_{i =1}^n R_n(z)_{ii}-\EE R_n(z)_{ii} \right|^2 \right]&=\frac{1}{n}\EE\left|R_n(z)_{11}-\EE R_n(z)_{11} \right|^2\\
	&\qquad+\frac{n-1}{n}\cov(R_n(z)_{11},R_n(z)_{22})\\
	&\leq \frac{1}{n\Im(z)^2}\\
	&\qquad+\frac{n-1}{n}\cov(R_n(z)_{11},R_n(z)_{22}).
\end{align*} From Theorems \ref{Thm:ResfromLC} and \ref{thm:LocalConv2} we know $R_{n}(z)_{11}$ and $R_n(z)_{22}$ are asymptotically independent random variables bounded uniformly in $n$, and thus asymptotically uncorrelated. From this we get \begin{equation*}
\lim\limits_{n\rightarrow\infty}\EE\left[\left|\frac{1}{n}\sum_{i =1}^n R_n(z)_{ii}-\EE R_n(z)_{ii} \right|^2 \right]=0,
\end{equation*} and \begin{equation}\label{eq:StieltjesConv}
\lim\limits_{n\rightarrow\infty}\EE|s_{L_n}(z)-\EE s_{\varnothing}(z)|=0.
\end{equation} Taking $\mu_m=\EE\mu_\varnothing$ completes the proof of Theorem \ref{thm:ESMConv}.

\section{Proof of theorem \ref{thm:RDEofStieltjes}}\label{sec:RDE}

We will follow the approach of \cite{SymHeavy} and take advantage of the tree structure on $\NN^f$ to arrive at \eqref{eq:RDEofStieltjes} before proving uniqueness. Let $L$ be the operator associated to $\mathbf{PWITL}(m)$, we have already seen that $s_m(z)=\EE s_\varnothing(z)$ where \begin{equation}
s_\varnothing(z)=\langle\delta_\varnothing,(L-z)^{-1}\delta_\varnothing\rangle.
\end{equation} We now decompose the operator $L$ as \begin{equation}\label{eq:OpDecomp}
L=C+\bigoplus_{k=1}^\infty L_k
\end{equation} where \begin{align}
\langle \delta_k,C\delta_\varnothing\rangle=\langle \delta_k,L\delta_\varnothing\rangle\nonumber\\
\langle \delta_k,C\delta_k\rangle=-y_k\\
\langle \delta_\varnothing,C\delta_\varnothing\rangle=\langle \delta_\varnothing,L\delta_\varnothing\rangle\nonumber
\end{align} and for every $k\in\NN$, $L_k$ is supported on $k\NN^f=\{kv\in\NN^f: v\in\NN^f \}$. Note $\langle \delta_v, C\delta_u\rangle=0$ for any other combination of $u,v\in\NN^f$. Under this decomposition $\{L_k\}_{k\geq 1}$ is a collection of i.i.d.\ random operators each equal in distribution, up to an isometry, to $L$. For convenience define the operator $\tilde{L}$ by \begin{equation}
\tilde{L}:=\bigoplus_{k=1}^\infty L_k,
\end{equation} and the operators $R(z):=(L-z)^{-1}$ and $\tilde{R}(z):=(\tilde{L}-z)^{-1}$ for $z\in\CC_+$. From \eqref{eq:OpDecomp} we get the resolvent identity \begin{equation}\label{eq:ResIdent}
\tilde{R}(z)CR(z)=\tilde{R}(z)-R(z).
\end{equation} Additionally denote by $R_{uv}(z):=\langle \delta_u, R(z)\delta_v\rangle$ and $\tilde{R}_{uv}(z):=\langle \delta_u, \tilde{R}(z)\delta_v\rangle$. Note $\tilde{R}_{\varnothing\varnothing}(z)=-z^{-1}$, $\tilde{R}_{kl}(z)=0$ for all $k,l\in\NN$ with $k\neq l$, and $\tilde{R}_{\varnothing k}(z)=0=\tilde{R}_{k\varnothing}(z)$ for all $k\in\NN$. 

From \eqref{eq:ResIdent} one immediately gets \begin{equation}
\langle \delta_k, \tilde{R}(z)CR(z)\delta_\varnothing\rangle=-R_{k\varnothing}(z).
\end{equation} It also follows that\begin{align*}
\langle \delta_k, \tilde{R}(z)CR(z)\delta_\varnothing\rangle&=\left\langle \delta_k, \tilde{R}(z)C\sum_{v\in\NN^f}R_{v\varnothing}(z)\delta_v\right\rangle\\
&=\left\langle \delta_k, \tilde{R}(z)R_{\varnothing\varnothing}(z)y_{\varnothing\varnothing}\delta_\varnothing\right\rangle+\left\langle \delta_k, \tilde{R}(z)\sum_{l\in\NN}R_{\varnothing\varnothing}(z)y_l\delta_l\right\rangle\\&\quad+\left\langle \delta_k, \tilde{R}(z)\sum_{j\in\NN}R_{j\varnothing}(z)y_j\delta_\varnothing\right\rangle
-\left\langle \delta_k, \tilde{R}(z)\sum_{j\in\NN}R_{j\varnothing}(z)y_j\delta_j\right\rangle\\
&=0+\tilde{R}_{kk}(z)y_kR_{\varnothing\varnothing}(z)+0-\tilde{R}_{kk}(z)y_kR_{k\varnothing}(z).
\end{align*} Rearranging we arrive at \begin{equation}\label{eq:FirstResI}
R_{k\varnothing}(z)=\frac{\tilde{R}_{kk}(z)y_kR_{\varnothing\varnothing}(z)}{\tilde{R}_{kk}y_k-1}.
\end{equation} A similar computation for $\langle \delta_\varnothing, \tilde{R}(z)CR(z)\delta_\varnothing\rangle$ gives \begin{equation}\label{eq:SecondResI}
\tilde{R}_{\varnothing\varnothing}(z)-R_{\varnothing\varnothing}(z)=\tilde{R}_{\varnothing\varnothing}(z)y_{\varnothing\varnothing}R_{\varnothing\varnothing}(z)+\sum_{j=1}^\infty \tilde{R}_{\varnothing\varnothing}(z)y_{j}R_{j\varnothing}(z).
\end{equation} Combining \eqref{eq:FirstResI} and \eqref{eq:SecondResI} gives \begin{align*}
\tilde{R}_{\varnothing\varnothing}(z)&=R_{\varnothing\varnothing}(z)+\tilde{R}_{\varnothing\varnothing}(z)y_{\varnothing\varnothing}R_{\varnothing\varnothing}(z)+\sum_{j=1}^\infty \tilde{R}_{\varnothing\varnothing}(z)y_{j}R_{j\varnothing}(z)\\
&=R_{\varnothing\varnothing}(z)+\tilde{R}_{\varnothing\varnothing}(z)y_{\varnothing\varnothing}R_{\varnothing\varnothing}(z)+\sum_{j=1}^\infty \tilde{R}_{\varnothing\varnothing}(z)y_{j}\frac{\tilde{R}_{jj}(z)y_jR_{\varnothing\varnothing}(z)}{\tilde{R}_{jj}y_j-1}\\
&=R_{\varnothing\varnothing}(z) \left[1+\tilde{R}_{\varnothing\varnothing}(z)y_{\varnothing\varnothing}+\sum_{j=1}^\infty \tilde{R}_{\varnothing\varnothing}(z)y_{j}\frac{\tilde{R}_{jj}(z)y_j}{\tilde{R}_{jj}y_j-1}\right],\\
\end{align*} which, along with $\tilde{R}_{\varnothing\varnothing}(z)=-z^{-1}$, implies\begin{equation}
R_{\varnothing\varnothing}=-\left(z-y_{\varnothing\varnothing}-\sum_{j=1}^\infty\frac{\tilde{R}_{jj}(z)y_j^2}{\tilde{R}_{jj}(z)y_j-1} \right)^{-1}.
\end{equation} Noting $y_{\varnothing\varnothing}=-\sum_{j=1}^\infty y_j$ gives \eqref{eq:RDEofStieltjes}. Note that for $j\in\NN$ $\tilde{R}_{jj}(z)$ depends only on $z$ and $L_j$, and hence $\{\tilde{R}_{jj}(z)\}_{j\in\NN}$ is a collection of i.i.d.\ random variables independent of $\{y_j\}_{j\in\NN}$.

\subsection{Uniqueness} In this section we prove uniqueness of the solution to \eqref{eq:RDEofStieltjes} from Theorem \ref{thm:RDEofStieltjes}. While the argument is technical, the core is a contraction approach. We will show the map $T$ defined below in \eqref{eq:MapDefinition} would contract, in an appropriate metric, two fixed points belonging to a nice subset of all probability measures on the space of Stieltjes transforms. We then extend this result to any two potential fixed points by moving from this metric to a functional separating distinct points. 

Let $\mathcal{S}$ be the set of Stieltjes transforms of probability measures on $\RR$ and $\mathcal{P}(\mathcal{S})$ be the set of probability measures on $\mathcal{S}$. Define $T:\mathcal{P}(\mathcal{S})\rightarrow\mathcal{P}(\mathcal{S})$ as follows: for $\mu\in\mathcal{P}(\mathcal{S})$\begin{equation}\label{eq:MapDefinition}
	T(\mu):=\mathcal{L}\left(-\left(z-\sum_{j=1}^N\frac{y_j}{s_j(z)y_j-1} \right)^{-1}\right),
\end{equation} 
where $\{s_j\}$ are i.i.d.\ with distribution $\mu$, $\{y_j\}_{j=1}^\infty$ is a Poisson point process with a fixed intensity measure $m$ independent of the collection $\{s_j\}_{j\geq1}$, $N$ is a Poisson random variable with mean $m(\RR)$ such that $y_j=0$ if $j>N$, and $\mathcal{L}(X)$ is the law of a random variable $X$.  Thus the distribution of $s_\varnothing$ is a fixed point of $T$ and we aim to show it is the unique fixed point. The notation of distance for which $T$ contracts fixed points will involve the infimum over all couplings of these fixed point measures. Let $\mu_1,\mu_2\in\mathcal{P}(\mathcal{S})$ be two fixed points of $T$ and let $(s(z),r(z))$ be an arbitrary coupling of $\mu_1$ and $\mu_2$. Additionally let $\mu_r$ and $\mu_s$ be the random probability measures on $\RR$ defined uniquely by\begin{equation*}
	s(z)=\int_\RR\frac{1}{x-z}d\mu_s(x)\text{ and }r(z)=\int_\RR\frac{1}{x-z}d\mu_r(x),
\end{equation*} for all $z\in\CC_+$. For now we will assume there exists $M\in\NN$ such that almost surely $\mu_r([-M,M])\geq\frac{1}{2}$ and $\mu_s([-M,M])\geq \frac{1}{2}$. This assumption will be removed later. As $r$ and $s$ are analytic functions on the upper half plane we will consider them only on the box \begin{equation}\label{eq:Set definition}
\mathcal{C}_M=\left\{z\in\CC: |\Re(z)|\leq\frac{1}{2},\, f_m(M)\leq\Im(z)\leq f_m(M)+1 \right\},
\end{equation} where $f_m$ is a positive increasing function on $\NN$ such that $f_m(M)\rightarrow\infty$ as $M\rightarrow\infty$, which will be chosen later to satisfy \eqref{eq:A:f_m def} below. Note for $z\in\mathcal{C}_M$, the assumption on $\mu_r$ and $\mu_s$ imply $\Im(r(z)),\Im(s(z))\geq\frac{1}{2}\frac{\Im(z)}{(M+\frac{1}{2})^2+\Im(z)^2}$.  Let $\{(r_j,s_j) \}_{j=1}^N$ be i.i.d.\ copies of $(s(z),r(z))$. Define the random functions $\tilde{s}$ and $\tilde{r}$, pointwise on $\CC_+$ and the sample space, by \begin{equation}\label{eq:specificcoupling}
\tilde{s}(z)=-\left(z-\sum_{j=1}^N\frac{y_j}{s_j(z)y_j-1} \right)^{-1}\text{ and }\tilde{r}(z)=-\left(z-\sum_{j=1}^N\frac{y_j}{r_j(z)y_j-1} \right)^{-1},
\end{equation} where $\{y_j\}_{j=1}^\infty$ is a Poisson point process with intensity measure $m$ independent of the collection $\{(r_j,s_j) \}_{j=1}^\infty$. If $\mu_1$ and $\mu_2$ are fixed points of $T$, then $(\tilde{r},\tilde{s})$ is a coupling of $\mu_1$ and $\mu_2$. We show that \begin{equation*}
\EE\sup_{z\in\mathcal{C}_M}|\tilde{r}(z)-\tilde{s}(z)|\leq  \frac{4}{5}\EE\sup_{z\in\mathcal{C}_M}|r(z)-s(z)|
\end{equation*} for an appropriate choice of $f_m(M)$ independent of the coupling $(r,s)$. First note \begin{align}\label{eq:UniquenessStep1}
\EE\sup_{z\in\mathcal{C}_M}|\tilde{r}(z)-\tilde{s}(z)|&\leq\EE\sup_{z\in\mathcal{C}_M}\frac{1}{\Im(z)^2}\sum_{j=1}^N\frac{|r_j(z)-s_j(z)|y_j^2}{|r_j(z)y_j-1||s_j(z)y_j-1|}\nonumber\\
			&\leq\EE\sum_{j=1}^N\sup_{z\in\mathcal{C}_M}\frac{1}{\Im(z)^2}\frac{|r_j(z)-s_j(z)|y_j^2}{|r_j(z)y_j-1||s_j(z)y_j-1|}.
\end{align} To handle the denominator we will consider separately the points where $\Re(r_j(z)y_j)$ is small and the few points where $\Im(r_j(z)y_j)$ is large. Let $\hat{m}$ be equal to $m$ with support restricted to $[-f_m(M)/2,f_m(M)/2]$ and $\tilde{m}:=m-\hat{m}$. Decompose the point process $\{y_j\}_{j=1}^N$ into two independent Poisson point processes $\{\hat{y}_j\}^{\hat{N}}_{j=1}$ and $\{\tilde{y}_j\}^{\tilde{N}}_{j=1}$ with intensity measures $\hat{m}$ and $\tilde{m}$ respectively. We will divide the sum in \eqref{eq:UniquenessStep1} into two sums over these point processes. To begin note for $z\in\mathcal{C}_M$, $|s_j(z)\hat{y}_j-1||r_j(z)\hat{y}_j-1|\geq1/4$ and thus \begin{align}
\EE\sum_{j=1}^{\hat{N}}\sup_{z\in\mathcal{C}_M}\frac{1}{\Im(z)^2}\frac{|r_j(z)-s_j(z)|\hat{y}_j^2}{|r_j(z)\hat{y}_j-1||s_j(z)\hat{y}_j-1|}&\leq\EE\sum_{j=1}^{\hat{N}}\sup_{z\in\mathcal{C}_M}\frac{4}{\Im(z)^2}|r_j(z)-s_j(z)|\hat{y}_j^2\nonumber\\
&\leq \frac{4}{f_m(M)^2}\EE\sum_{j=1}^{\hat{N}}\sup_{z\in\mathcal{C}_M}|r_j(z)-s_j(z)|\hat{y}_j^2\nonumber\\
&=\frac{4}{f_m(M)^2}\EE\sup_{z\in\mathcal{C}_M}|r_1(z)-s_1(z)|I_{M,m},\nonumber
\end{align} where $I_{M,m}=\int_{-f_m(M)/2}^{f_m(M)/2}y^2dm(y)$ and the last equality follows from Lemma \ref{lemma:CampbellsFormula} and independence. From \eqref{eq:MeasureDecayAssumption}\begin{align*}
\int_{-f_m(M)/2}^{f_m(M)/2}y^2dm(y)&\leq \int_{-1}^1 y^2\ dm(y)+\int_{1\leq|x|\leq\sqrt{f_m(M)}}y^2\ dm(y)\\
								   &\qquad+ \int_{\sqrt{f_m(M)}\leq|x|\leq f_m(M)/2}y^2\ dm(y)\\
								   &\leq C'+C'f_m(M)+ \frac{f_m(M)^2}{4}m(\{x: |x|\geq \sqrt{f_m(M)}\})\\
								   & \leq Cf_m(M)^{2-\ee/2}
\end{align*} where $\ee>0$ is from \eqref{eq:MeasureDecayAssumption}, and $C,C'>0$ are constants which depend only on the measure $m$. Thus \begin{equation}\label{eq:HatContraction}
\EE\sum_{j=1}^{\hat{N}}\sup_{z\in\mathcal{C}_M}\frac{1}{\Im(z)^2}\frac{|r_j(z)-s_j(z)|\hat{y}_j^2}{|r_j(z)\hat{y}_j-1||s_j(z)\hat{y}_j-1|}\leq\frac{C'}{f_m(M)^{\ee/2}}\EE\sup_{z\in\mathcal{C}_M}|r_1(z)-s_1(z)|.
\end{equation}
To handle the other sum first note $|s_j(z)\tilde{y}_j-1||r_j(z)\tilde{y}_j-1|\geq \tilde{y}_j^2 C_{z,M}^{-1}$ where\begin{equation*}
C_{z,M}=\frac{4\left(\left(M+\frac{1}{2}\right)^2+\Im(z)^2\right)^2}{\Im(z)^2}.
\end{equation*} Then \begin{align}\label{eq:TildeContration}
\EE\sum_{j=1}^{\tilde{N}}\sup_{z\in\mathcal{C}_M}\frac{1}{\Im(z)^2}\frac{|r_j(z)-s_j(z)|\tilde{y}_j^2}{|r_j(z)\tilde{y}_j-1||s_j(z)\tilde{y}_j-1|}&\leq\EE\sum_{j=1}^{\tilde{N}}\sup_{z\in\mathcal{C}_M}\frac{C_{z,M}}{\Im(z)^2}|r_j(z)-s_j(z)|\nonumber\\
&\leq\frac{C_{i(f_m(M)+1),M}}{f_m(M)^2}\sum_{j=1}^{\tilde{N}}\sup_{z\in\mathcal{C}_M}|r_j(z)-s_j(z)|\nonumber \\
&= \frac{C_{i(f_m(M)+1),M}}{f_m(M)^2}\EE\tilde{N}\EE\sup_{z\in\mathcal{C}_M}|r_1(z)-s_1(z)|, 
\end{align} where the final equality follows from Lemma \ref{lemma:CampbellsFormula}. Finally combining \eqref{eq:HatContraction} and \eqref{eq:TildeContration} gives \begin{equation}\label{eq:ArbContraction}
\EE\sup_{z\in\mathcal{C}_M}|\tilde{r}(z)-\tilde{s}(z)|\leq  \left(\frac{C'}{f_m(M)^{\ee/2}}+\frac{C_{i(f_m(M)+1),M}}{f_m(M)^2}\EE\tilde{N} \right)\EE\sup_{z\in\mathcal{C}_M}|r_1(z)-s_1(z)|.
\end{equation} Notice this coefficient is independent of the coupling and depends only on $M$, $f_m$ and $m$. From the definition of $C_{z,M}$ we have that $C_{z+i,M}/\Im(z)^2\rightarrow4$ as $\Im(z)\rightarrow\infty$. We also have that $\EE \tilde{N}=m\left((-\infty,-f_m(M)/2)\cup(f_m(M)/2,\infty) \right)\leq Cf_m(M)^{-\epsilon}$. We choose $f_m$ to be such that \begin{equation}\label{eq:A:f_m def}
\left(\frac{C'}{f_m(M)^{\ee/2}}+\frac{C_{i(f_m(M)+1),M}}{f_m(M)^2}\EE\tilde{N} \right)\leq\frac{4}{5},
\end{equation} for each $M\in\NN$. As the left hand side of \eqref{eq:A:f_m def} is decreasing in $f_m(M)$, $f_m$ may be chosen to be increasing and unbounded.

Next we remove that assumption that, for some $M$, almost surely $\mu_r$ and $\mu_s$ have half their mass in $[-M,M]$. For a positive, increasing, unbounded function $f_m$ on $\NN$, we define the function $d_{f_m}:\mathcal{P}(\mathcal{S})^2\rightarrow[0,\infty)$ by\begin{align}\label{eq:metric d}
	d_{f_m}(\mu_1,\mu_2)=\inf_{(r,s)\in\mathcal{C}(\mu_1,\mu_2)}\EE\sum_{M=1}^{\infty}\sup_{z\in\mathcal{C}_M}|r(z)-s(z)|/2^M, 
\end{align} and the function $d'_{f_m}:\mathcal{P}(\mathcal{S})^2\rightarrow[0,\infty)$ by \begin{align*}
	d'_{f_m}(\mu_1,\mu_2)=\inf_{(r,s)\in\mathcal{C}(\mu_1,\mu_2)}\EE\sum_{M=1}^{\infty}\sup_{z\in\mathcal{C}_M}|r(z)-s(z)|\oindicator{A_M}/2^M,
\end{align*} where $\oindicator{A_M}$ is the indicator function of the event \begin{equation*}
A_M=\{\mu_r([-M,M])\geq\frac{1}{2} \text{ and } \mu_s([-M,M])\geq \frac{1}{2} \},
\end{equation*} $\mathcal{C}_M$ is the set defined by \eqref{eq:Set definition}, and $\mathcal{C}(\mu_1,\mu_2)$ is the set of all couplings of $\mu_1$ and $\mu_2$ for $\mu_1,\mu_2\in \mathcal{P}(\mathcal{S})$. It is straightforward to check that $\rho:\mathcal{S}\times\mathcal{S}\rightarrow[0,\infty)$ defined by\begin{equation}
\rho(s,r)=\sum_{M=1}^{\infty}\sup_{z\in\mathcal{C}_M}|r(z)-s(z)|/2^M
\end{equation} is a metric on $\mathcal{S}$, and thus $d_{f_m}$ is the $1^{\text{st}}$-Wasserstein metric on $\mathcal{P}(\mathcal{S})$ (see \cite{DudleyBook} Chapter 11 for details). Let $\mu_1$ and $\mu_2$ be two fixed points of $T$. Let $(s,r)$ be a coupling of $\mu_1$ and $\mu_2$ such that \begin{equation}\label{eq:10/9 coupling}
\EE\sum_{M=1}^{\infty}\sup_{z\in\mathcal{C}_M}|r(z)-s(z)|\oindicator{A_M}/2^M\leq\frac{10}{9}d'_{f_m}(\mu_1,\mu_2),
\end{equation} and let $\tilde{s}$ and $\tilde{r}$ be built from i.i.d.\ copies of $(s,r)$ as in \eqref{eq:specificcoupling}. Using the specific coupling $(\tilde{s},\tilde{r})$,  \eqref{eq:ArbContraction}, and \eqref{eq:10/9 coupling} we get \begin{align*}
d'_{f_m}(\mu_1,\mu_2)&\leq\EE\sum_{M=1}^{\infty}\sup_{z\in\mathcal{C}_M}|\tilde{r}(z)-\tilde{s}(z)|\oindicator{A_M}/2^M\\
			   &\leq\frac{4}{5}\EE\sum_{M=1}^{\infty}\sup_{z\in\mathcal{C}_M}|r(z)-s(z)|\oindicator{A_M}/2^M\\
			   &\leq\frac{8}{9}d'_{f_m}(\mu_1,\mu_2),
\end{align*} and thus $d'_{f_m}(\mu_1,\mu_2)=0$.

If $d'_{f_m}$ was a metric it would be immediate that $\mu_1=\mu_2$, however it is not clear this is the case. The only property of a metric needed is that $d'_{f_m}$ separates distinct points in $\mathcal{P}(\mathcal{S})$, and thus we conclude the proof using the following lemma.

\begin{lemma}\label{lemma:d'=0iffd=0}
	Fix a positive, increasing, unbounded function $f_m:\NN\rightarrow\RR$. $d'_{f_m}(\mu_1,\mu_2)=0$ if and only if $d_{f_m}(\mu_1,\mu_2)=0$ for the metric $d_{f_m}$ defined in \eqref{eq:metric d}.
\end{lemma}

\begin{proof}

Assume $d'_{f_m}(\mu_1,\mu_2)=0$, fix $\ee>0$, and note there exists $M_0\in\NN$ such that for any $M\geq M_0$ and any coupling $(r,s)$ one has \begin{equation}\label{eq:A_Mlarge}
\PP(A_M)\geq 1-\ee.
\end{equation} We have that \begin{equation}
\inf_{(r,s)\in\mathcal{C}(\mu_1,\mu_2)}\EE\sup_{z\in\mathcal{C}_{M_0}}|r(z)-s(z)|\oindicator{A_{M_0}}=0,
\end{equation} and thus we can find a sequence of couplings $\{(r_n,s_n)\}_{n=1}^\infty$ such that \begin{equation}\label{eq:First set convergence}
\lim\limits_{n\rightarrow\infty}\EE\sup_{z\in\mathcal{C}_{M_0}}|r_n(z)-s_n(z)|\oindicator{A_{M_0}^n}=0,
\end{equation} where $A_{M_0}^n=\{\mu_r^n([-M,M])\geq\frac{1}{2} \text{ and } \mu_s^n([-M,M])\geq \frac{1}{2} \}$ and $\mu_r^n$ and $\mu_s^n$ are the random probability measures associated to $r_n$ and $s_n$. Let $M_1$ be such that $f_m(M_1)>\frac{1}{2\eps}$, and hence for any $M\geq M_1$ \begin{equation}\label{eq:High sets don't matter}
\sup_{z\in\mathcal{C}_M}|r(z)-s(z)|\leq \eps,
\end{equation} for any Stieltjes transforms $r$ and $s$. 

We will now extend the convergence in \eqref{eq:First set convergence} to the supremum over the larger compact set $\tilde{\mathcal{C}}=\cup_{j=1}^{M_1}\mathcal{C}_j$. The $L^{1}$-convergence of the random variables $\sup_{z\in\mathcal{C}_{M_0}}|r_n(z)-s_n(z)|\oindicator{A_{M_0}^n}$ in \eqref{eq:First set convergence} to zero implies convergence in probability to zero. Thus we can find a subsequence converging almost surely to zero, and without loss of generality we denote this subsequence $\{\sup_{z\in\mathcal{C}_{M_0}}|r_n(z)-s_n(z)|\oindicator{A_{M_0}^n}\}_{n=1}^\infty$. Let $G=\{\lim_{n\rightarrow\infty}\sup_{z\in\mathcal{C}_{M_0}}|r_n(z)-s_n(z)|\oindicator{A_{M_0}^n}=0\}$, and decompose $G$ into $G_1=\{\omega\in G: \oindicator{A_{M_0}^n}(\omega)=1 \text{ i.o.}\}$ and $G_0=G\setminus G_1$. Clearly on $G_0$ the random variables $\sup_{z\in\tilde{\mathcal{C}}}|r_n(z)-s_n(z)|\oindicator{A_{M_0}^n}$ are eventually identically $0$. For  $\omega\in G_1$ we consider the further subsequence $\{n_k\}$ such that $\oindicator{A_{M_0}^{n_k}}(\omega)=1$ for all $k$. For this outcome $\omega$ we have $\{(r_{n_k}-s_{n_k})(\omega)\}_{n=1}^\infty$ is a sequence of complex analytic functions on $\CC_+$, uniformly bounded on compact subsets of $\CC_+$, converging uniformly to $0$ on a set with an accumulation point. Thus applying the Vitali convergence theorem for analytic functions, Lemma \ref{lemma:VTC}, we get that $\sup_{z\in\tilde{\mathcal{C}}}|(r_{n_k}(z)-s_{n_k}(z))|(\omega)\rightarrow 0$ as $n_k\rightarrow\infty$. From the above and the bounded convergence theorem we get \begin{equation}\label{eq:ConvonAllsets}
\lim\limits_{n\rightarrow\infty}\EE\sup_{z\in\tilde{\mathcal{C}}}|r_n(z)-s_n(z)|\oindicator{A_{M_0}^n}=0.
\end{equation} Combining \eqref{eq:A_Mlarge}, \eqref{eq:High sets don't matter}, and \eqref{eq:ConvonAllsets}, we obtain \begin{align*}
d_{f_m}(\mu_1,\mu_2)&=\inf_{(r,s)\in\mathcal{C}(\mu_1,\mu_2)}\EE\sum_{M=1}^{\infty}\sup_{z\in\mathcal{C}_M}|r(z)-s(z)|/2^M\\
&\leq\inf_{(r,s)\in\mathcal{C}(\mu_1,\mu_2)}\EE\sum_{M=1}^{\infty}\left(\sup_{z\in\mathcal{C}_M}|r(z)-s(z)|\oindicator{A_{M_0}}/2^M+\frac{2}{f_m(1)}\PP(A_{M_0}^c)/2^M\right)\\
&<\inf_{(r,s)\in\mathcal{C}(\mu_1,\mu_2)}\EE\sum_{M=1}^{M_1}\sup_{z\in\tilde{\mathcal{C}}}|r(z)-s(z)|\oindicator{A_{M_0}}/2^M+\ee+\frac{2}{f_m(1)}\eps\\
&\leq\inf_{(r,s)\in\mathcal{C}(\mu_1,\mu_2)}\EE\sup_{z\in\tilde{\mathcal{C}}}|r(z)-s(z)|\oindicator{A_{M_0}}+\left(1+\frac{2}{f_m(1)}\right)\ee\\
&=\left(1+\frac{2}{f_m(1)}\right)\ee.
\end{align*} As $\ee>0$ was arbitrary we have $d_{f_m}(\mu_1,\mu_2)=0$. For the other direction note \begin{align*}
d'_{f_m}(\mu_1,\mu_2)&=\inf_{(r,s)\in\mathcal{C}(\mu_1,\mu_2)}\EE\sum_{M=1}^{\infty}\sup_{z\in\mathcal{C}_M}|r(z)-s(z)|\oindicator{A_M}/2^M\\
			   &\leq\inf_{(r,s)\in\mathcal{C}(\mu_1,\mu_2)}\EE\sum_{M=1}^{\infty}\sup_{z\in\mathcal{C}_M}|r(z)-s(z)|/2^M\\
			   &=d_{f_m}(\mu_1,\mu_2),
\end{align*} for any $\mu_1,\mu_2$.
\end{proof}



\appendix

\section{Tightness of $\mu_{L_n}$} 

\begin{lemma}\label{lemma:tightness}
	Let $\{A_n\}_{n\geq 1}$ be a L\'evy-Khintchine random matrix ensemble with characteristics $(0,b,m)$ satisfying \textbf{C1} and $L_n$ be the matrix defined by \eqref{eq:GenLaplacianDef}. Then there exists $r>0$ such that almost surely \begin{equation*}
	\limsup\limits_{n\rightarrow\infty}\int_\RR |t|^rd\mu_{L_n}(t)<\infty,
	\end{equation*} and thus almost surely $\{\mu_{L_n}\}_{n\geq 1}$ is tight.
\end{lemma}

\begin{proof}
	This is essentially an extension of the argument in Lemma B.3 of \cite{SymHeavy} to L\'evy-Khintchine matrices and with the matrix $-D_n$ added. Applying Lemma \ref{Basics} and noting $|\lambda_i(A)|=s(A)$ for some singular value of a Hermitian matrix $A$ one gets for any $0\leq k\leq n-1$\begin{equation*}
	|\lambda_{1+k}(L_n)|\leq|\lambda_{1+\lfloor k/2\rfloor}(A_n)|+|\lambda_{1+\lceil k/2\rceil}(-D_n)|.
	\end{equation*} Thus \begin{equation}
	\int_\RR |t|^rd\mu_{L_n}(t)\leq 8\left(\int_\RR |t|^rd\mu_{A_n}(t)+\int_\RR |t|^rd\mu_{D_n}(t) \right).
	\end{equation} $D_n$ is a diagonal matrix so \begin{equation}
	\int_\RR|t|^rd\mu_{D_n}(t)=\frac{1}{n}\sum_{k =1}^n\left|\sum_{j\in\{1,\dots,n\}\setminus\{k\}}A_{kj}^{(n)} \right|^r\leq\frac{1}{n}\sum_{k =1}^n\left|\sum_{j\in\{1,\dots,n\}\setminus\{k\}}|A_{kj}^{(n)}| \right|^r.
	\end{equation} Assuming $0\leq r\leq 2$ and applying the Schatten Bound, Lemma \ref{Schattent}, to $A_n$ we get \begin{equation}
	\int_\RR|t|^rd\mu_{A_n}(t)\leq\frac{1}{n}\sum_{k =1}^n\left|\sum_{j\in\{1,\dots,n\}\setminus\{k\}}|A_{kj}^{(n)}|^2 \right|^{r/2}.
	\end{equation} We will prove that almost surely \begin{equation}
	\limsup\limits_{n\rightarrow\infty}\int_\RR|t|^rd\mu_{D_n}(t)<\infty.
	\end{equation} The proof for $\mu_{A_n}$ follows with only minor changes. Both follow the arguments of Lemma B.1 in \cite{SymHeavy}. Define the random variable $Y_{n,k}$ by \begin{equation*}
	Y_{n,k}:=\left|\sum_{j\in\{1,\dots,n\}\setminus\{k\}}|A_{kj}^{(n)}| \right|^r,
	\end{equation*} for all $k\in\{1,\dots,n\}$. From the proof of Lemma B.1 in \cite{SymHeavy} we see it is enough to show \begin{equation}
	\sup_{n\geq 1}\EE Y_{n,1}^4<\infty.
	\end{equation} Define for any $0\leq a<b$\begin{equation*}
	S_{n,a,b}:=\sum_{j=2}^n |A_{1j}^{(n)}|\indicator{|A_{1j}^{(n)}|\in[a,b)}.
	\end{equation*} Then $Y^4_{n,1}=S_{n,0,\infty}^{4r}=(S_{n,0,1/2}+S_{n,1/2,\infty})^{4r}$ and \begin{equation*}
	\EE Y^4_{n,1}\leq 2^{4r-1}\left(\EE S_{n,0,1/2}^{4r}+\EE S_{n,1/2,\infty}^{4r} \right).
	\end{equation*} If we further assume $4r<1$ we can apply Jensen's inequality to get\begin{equation*}
	\EE S_{n,0,1/2}^{4r}\leq (\EE S_{n,0,1/2})^{4r}.
	\end{equation*} Applying Proposition \ref{prop:IDconv} to the triangular array $\{|A_{1j}^{(n)}|\}$ we get that \begin{equation*}
	(\EE S_{n,0,1/2})^{4r}\rightarrow\hat{b}_{1/2}.
	\end{equation*} as $n\rightarrow\infty$. For the larger entries of the row\begin{equation*}
	\EE S_{n,1/2,\infty}^{4r}\leq\EE[(N_nM_n)^{4r}]
	\end{equation*} where 
	$$M_n:=\max_{2\leq j\leq n}|A_{1j}^{(n)}|,$$
	and 
	$$N_n:=|\{2\leq j\leq n: |A_{1j}^{(n)}|\geq1/2 \}|,$$
	where $|G|$ is the cardinality of a set $G$. Again, using Proposition \ref{prop:IDconv} we get for sufficiently large $n$ and any $k\in\NN$\begin{align*}
	\PP(N_n= k)&\leq {n\choose k}\PP(|A_{12}^{(n)}|\geq1/2)^k\\
				  &=\frac{{n\choose k} }{n^k}\left(n\PP(|A_{12}^{(n)}|\geq 1/2) \right)^k\\
				  &\leq\frac{{n\choose k} }{n^k}c^k\\
				  &\leq\frac{c^k}{k!}
	\end{align*} where $c=m((-\infty,1/2]\cup[1/2,\infty))+1$, and we see\begin{align*}
	\PP(N_n\geq k)\leq e^c-\sum_{j=1}^{k-1}\frac{c^j}{j!}\leq\frac{Cc^k}{k!}.
\end{align*} Thus $\sup_{n\geq 1}\EE N_n^{\eta}<\infty$ for any $\eta>0$. For $M_n$ note that from \eqref{eq:Entry Tail Assumption}\begin{equation*}
	\PP(M_n\geq t)\leq n\PP(|A_{12}^{(n)}|\geq t)\leq Ct^{-\ee}.
	\end{equation*} It follows that $\sup_{n\geq 1}\EE M_n^{\gamma}<\infty$ for any $0<\gamma<\ee$. Returning to $S_{n,1/2,\infty}$ and applying H\"older's inequality we get \begin{equation*}
	\sup_{n\geq 1}\EE S_{n,1/2,\infty}^{4r}\leq\sup_{n\geq 1}\EE[(N_nM_n)^{4r}]\leq\sup_{n\geq 1}(\EE N_n^{4rp})^{1/p}(\EE M_n^{4rq})^{1/q}<\infty
	\end{equation*}for small enough $q$, completing the proof.
\end{proof}

\section{Additional lemmas}\label{sec:Appendix}

\begin{lemma}\label{lemma:Tau<1}
	Let $m$ be a measure on $\RR\setminus\{0\}$ such that \begin{equation}
		\int_{\RR}1\wedge |x|dm(x)<\infty,
	\end{equation} and let $\{y_k\}_{k= 1}^N$ be a Poisson point process with intensity measure $m$ where $N=\infty$ a.s.\ if $m(\RR\setminus\{0\})=\infty$. If $\tau_\kappa=\inf\left\{t\geq 1:\sum_{k =t+1}^N y_{(k)}^2\leq\kappa \right\}$ where $y_{(1)}\geq y_{(2)}\geq\dots$ is a non-increasing ordering of $\{y_k\}_{k= 1}^N$, then $\EE\tau_\kappa<\infty$ for all $\kappa>0$ and $\EE\tau_\kappa\rightarrow0$ as $\kappa\rightarrow\infty$.
\end{lemma}

\begin{proof}
	The fact that $\tau_\kappa$ is almost surely finite follows from the integrability condition on $m$. Additionally $\PP(\tau_\kappa=0)=\PP\left(\sum_{k =1}^N y_k^2\leq\kappa\right)$ and clearly converges to $1$ as $\kappa\rightarrow\infty$. Thus it is sufficient to prove $\EE\tau_\kappa<\infty$ for all $\kappa> 0$. Let $0<r_t<1$ be some monotonically decreasing function of $t$ such that $r_t\rightarrow0$ as $t\rightarrow\infty$, $S_t=\sum_{k=1}^N y_k^2\indicator{|y_k|\leq r_t}$, and $\Pi=\sum_{k =1}^N\delta_{y_j}$. Define the event $A_t=\{\Pi((-\infty,-r_t)\cup(r_t,\infty))\geq t\}$. On $A_t^c$ the collection of points summed over in the definition of $\tau_\kappa$ is a strict subset of the collection of points summed over in the definition of $S_t$.  Then\begin{align*}
		\{\tau_\kappa>t\}&\subseteq\left(\{\tau_\kappa>t\}\cap A_t^c\right)\cup\left(\{\tau_\kappa>t\}\cap A_t\right) \\
						 &\subseteq\left(\{S_t\geq\kappa\}\cap A_t^c\right)\cup A_t \\
						 &\subseteq\{S_t\geq\kappa\}\cup A_t.
	\end{align*} We will now show for appropriate $r_t$ the probabilities of the events $\{S_t\geq\kappa\}\cup\{\Pi((-\infty,-r_t)\cup(r_t,\infty))\geq t\}$ are summable in $t\in\NN$.

	For a Poisson random variable $X$ with mean $\lambda$, $\PP(X\geq t)\leq\exp(-t\log(\frac{t}{\lambda e}))$. Letting $E_t=\EE\Pi((-\infty,-r_t)\cup(r_t,\infty))=m((-\infty,-r_t)\cup(r_t,\infty))<\infty$ we have\begin{equation}
		\PP\left(\Pi((-\infty,-r_t)\cup(r_t,\infty))\geq t\right)\leq\exp\left(-t\log\left(\frac{t}{E_te}\right)\right).
	\end{equation} For the event $\{S_t\geq\kappa\}$ notice that \begin{equation*}
	\PP(S_t\geq\kappa)\leq\exp(-\theta\kappa)\EE\exp(\theta S_t),
\end{equation*} and from Campbell's Formula, Lemma \ref{lemma:CampbellsFormula}, \begin{align*}
\EE\exp(\theta S_t)&=\exp\left(\int_{-r_t}^{r_t}(e^{\theta x^2}-1)dm(x) \right)\\
				   &\leq C\exp\left(\theta\int_{-r_t}^{r_t}x^2dm(x) \right).
\end{align*} Letting $\theta=\theta_t =\left(\int_{-r_t}^{r_t}x^2dm(x) \right)^{-1}$ the above gives us\begin{equation}\label{eq:taubound}
\PP(\tau_\kappa>t)\leq\exp\left(-t\log\left(\frac{t}{E_te}\right)\right)+C'\exp(-\theta_t\kappa).
\end{equation} Notice the integrability assumption on $m$ implies\begin{equation}\label{eq:m not too big}
	\ee m\left((-\infty,-\ee)\cup(\ee,\infty) \right)\leq C
\end{equation}  for all $\ee\in(0,1)$ and some constant $C$. From the integrability assumption we also have that\begin{equation}
\int_{-1}^1 |x|\,dm(x)=I
\end{equation} for some $0<I<\infty$. From this we see that\begin{equation}
\int_{-r_t}^{r_t}x^2\,dm(x)\leq r_tI,
\end{equation} and $\theta_t\geq (Ir_t)^{-1}$. Taking $r_t=\frac{c}{t}$ for some $c>0$ gives $\theta_t\geq (cI)^{-1}t$. For an appropriate choice of $c>0$ \eqref{eq:m not too big} and the definition of $E_t$ imply that $E_t\leq t/2e$. Thus from \eqref{eq:taubound} we get \begin{equation}
\EE\tau_\kappa=\sum_{t=0}^{\infty}\PP(\tau_\kappa\geq t)<\infty.
\end{equation} This completes the proof.\end{proof}

\begin{lemma}[Campbell's Formula, Section 3.2 in\cite{Kingman}]\label{lemma:CampbellsFormula}
	Let $\Pi$ be a Poisson point process on a measurable space $(\mathbb{X},\mathcal{X})$ with intensity measure $m$. Let $u:\mathbb{X}\rightarrow\RR$ be a measurable function. Then\begin{equation}
		S=\int_\mathbb{X}u(x)d\Pi(x)<\infty
	\end{equation} if and only if\begin{equation}
	\int_\mathbb{X}\min(u(x),1)dm(x)<\infty.
\end{equation} If either of the above integrals are finite then\begin{equation}
\EE\exp(\theta S)=\exp\left(\int_Xe^{\theta u(x)-1}dm(x) \right)
\end{equation} for any $\theta\in\CC$ for which the integral on the right hand side is finite. Moreover,\begin{equation}
		\EE\int_\mathbb{X}u(x)d\Pi(x)=\int_\mathbb{X}u(x)dm(x),
\end{equation} whenever $u\geq 0$ or $\int |u(x)|dm(x)<\infty$.
\end{lemma}

For $0<h<1$ define 
$$\sigma^2_h:=\sigma^2+\int_{|x|\leq h}x^2dm(x),\quad\text{ and }\quad b_h:=b-\int_{h<|x| }\frac{x}{1+x^2}dm(x).$$

\begin{proposition}[Corollary 15.16 in \cite{Kal02}]\label{prop:IDconv}
	Suppose $\{X_{ni}:1\leq i\leq n\}_{n\geq 1}$ is a triangular array of random variables such that each row consists of i.i.d.\ random variables. Then the sum
	$$\sum_{i =1}^n X_{ni},$$
	converges in distribution to an infinitely divisible random variable with characteristic $(\sigma^2,b,m)$ as $n\rightarrow\infty$ if and only if for every $0<h<1$ which is not an atom of $m$\begin{itemize}
		\item $n\PP(X_{n1}\in\cdot)\Rightarrow m(\cdot)$ on $\overline{\RR}\setminus\{0\}$,
		\item $n\EE\left[X_{n1}^2\indicator{|X_{n1}|\leq h}\right]\rightarrow\sigma_h^2$, and
		\item $n\EE\left[X_{n1}\indicator{|X_{n1}|\leq h} \right]\rightarrow b_h$,
	\end{itemize} as $n\rightarrow\infty$.
\end{proposition}

\begin{proposition}[Theorem 5.3 in \cite{HeavyPhenom}]\label{OrderStatConv}
	Suppose $\{X_{ni}:1\leq i\leq n\}_{n\geq 1}$ is a triangular array of random variables on $\overline{\RR}\setminus\{0\}$  such that each row consists of i.i.d.\ random variables. Let $N$ be a Poisson point process with intensity measure $m$. Then 
	$$\sum_{i =1}^n\delta_{X_{ni}}\Rightarrow N$$
	as $n\rightarrow\infty$ if and only if 
	$$n\PP(X_{n1}\in\cdot)\Rightarrow m(\cdot)$$
	as $n\rightarrow\infty$.
\end{proposition} 

\begin{lemma}[Vitali's convergence theorem for analytic functions, Lemma 2.14 in \cite{BaiSilverstein}]\label{lemma:VTC}
	Let $f_1,f_2,\dots$ be analytic in $D$, a connected open set of $\CC$, satisfying $|f_n(z)|\leq M$ for every $n$ and $z\in D$, and $f_n(z)$ converges as $n\rightarrow\infty$ for each $z$ in a subset of $D$ having an accumulation point in $D$. Then there exists a function $f$ analytic in $D$ for which $f_n(z)\rightarrow f(z)$ for all $z\in D$. Moreover on any set bounded by a contour interior to $D$, the convergence is uniform.
\end{lemma} Though Stieltjes transforms are not uniformly bounded on $\CC_+$, it is straightforward to apply Theorem \ref{lemma:VTC} to them by considering first $\CC_{+,m}=\{z\in\CC:\Im(z)>1/m\}$ and letting $m\rightarrow\infty$. 

\begin{lemma}[Theorem B.9 in \cite{BaiSilverstein}]\label{lemma:StieltjestoVague}
	Assume that $\{\mu_n\}$ is a sequence of functions probability measure, with Stieltjes transforms $\{s_n\}$. Then,\begin{equation}
		\lim_{n\rightarrow\infty}s_n(z)=s(z)
	\end{equation} for all $z\in\CC_+$ if and only if there exists a positive measure $\mu$ with Stieltjes transform $s$ such that $\mu_n$ converges to $\mu$ vaguely.
\end{lemma}

For the following lemma, we use the notation of \cite{Shirai}. For a connected open domain $D \subset \CC$, let $\mathcal{H}(D)$ be the space of analytic functions on $D$ equipped with the topology of uniform convergence on compact subsets of $D$.

\begin{lemma}[Proposition 2.5 in \cite{Shirai}]\label{tightness of random analytic functions}
	Let $X_1, X_2, X_3,\dots$ be a sequence of random analytic functions on a connected open set $D\subset\CC$, with probability  distribution measures $\mu_{X_1},\mu_{X_2},\dots$ on $\mathcal{H}(D)$. If for every $K\subset D$ compact $\{\max_{z\in K}|X_n(z)| \}_{n\geq 1}$ is a tight sequence of random variables, then $\{\mu_{X_n}\}_{n\geq 1}$ is tight in the space of probability measures on $\mathcal{H}(D)$.
\end{lemma}

\begin{lemma}[Schatten Bound, see proof of Theorem 3.32 in \cite{Matrixineq}]\label{Schattent}
	Let $A$ be an $n\times n$ complex Hermitian matrix with rows $R_1,\dots,R_n$. Then for every $0<r\leq 2$,
	$$\sum_{k=1}^{n}|\lambda_k(A)|^r\leq\sum_{k =1}^n\|R_k\|_{2}^r,$$
\end{lemma} where $\|\cdot \|_2$ is the Euclidean norm on $\CC^n$. 

\begin{lemma}[See \cite{HornTopics}, Chapter 3]\label{Basics}
	If $A$ and $B$ are $n\times n$ complex matrices and $s_1(A)\geq s_2(A)\geq\dots\geq s_n(A)$ and $s_1(B)\geq s_2(B)\geq\dots\geq s_n(B)$ are the singular values of $A$ and $B$ then 
	$$s_1(AB)\leq s_1(A)s_1(B)\text{ and }s_1(A+B)\leq s_1(A)+s_1(B),$$
	$$\max_{1 \leq k \leq n}|s_k(A+B)-s_k(A)|\leq s_1(B),$$
	$$s_{i+j-1}(A+B)\leq s_i(A)+s_j(B)$$
	for $1\leq i,j\leq n$ and $i+j\leq n+1$.
\end{lemma}

\begin{lemma}[See \cite{HeavyIId} Lemma C.1 and \cite{feller-vol-2} Theorem VIII.9.2]\label{TruncatedMoments}
	Let $Z$ be a positive random variable such that for every $t>0$, 
	$$\PP(Z\geq t)=L(t)t^{-\alpha}$$
	for some slowly varying function $L$ and some $\alpha\in(0,2)$.
	Then for every $p>\alpha$,
	$$\lim\limits_{t\rightarrow\infty}\frac{\EE[Z^p\indicator{Z\leq t}]}{ c(p)L(t)t^{p-\alpha}}=1,$$	
	where $c(p):=\alpha/(p-\alpha)$.
\end{lemma}

\bibliography{heavylaplacian}
\bibliographystyle{abbrv}

\end{document}